\documentclass[10pt,oneside,reqno]{amsart}

\usepackage[utf8]{inputenc}
\usepackage[T1]{fontenc}
\usepackage[russian,ngerman,english]{babel}
\usepackage{lmodern}

\usepackage{hyperref}
\usepackage{amsmath}
\usepackage{amsthm}
\usepackage{amsfonts}
\usepackage{amssymb}
\usepackage{amscd}
\usepackage{amsbsy}

\usepackage{pdfpages}
\usepackage{fancyhdr}
\usepackage{geometry}
\usepackage{setspace}
\usepackage{xcolor}
\usepackage{pifont}

\usepackage{graphicx}
\usepackage{tabularx}
\usepackage{epsfig}
\usepackage[all]{xy}
\usepackage{tikz}
\usetikzlibrary{matrix}
\usetikzlibrary{patterns.meta}

\usepackage{fixmath}

\usepackage{appendix}

\usepackage{enumerate}

\usepackage[sort, numbers]{natbib}
\usepackage{bibentry}

\usepackage{arydshln}

\newtheorem{theorem}{Theorem}[section]

\newtheorem{lemma}[theorem]{Lemma}
\newtheorem{proposition}[theorem]{Proposition}

\newtheorem{corollary}[theorem]{Corollary}

\theoremstyle{definition}
\newtheorem{definition}[theorem]{Definition}

\newtheorem{remark}{Remark}

\numberwithin{equation}{section}

\newcommand{\N}{{\mathbb N}}
\newcommand{\Z}{{\mathbb Z}}

\newcommand{\R}{{\mathbb R}}
\renewcommand{\C}{{\mathbb C}}

\newcommand{\A}{{\mathbb A}}

\newcommand{\E}{{\mathsf E}}

\DeclareMathOperator{\supp}{supp}
\DeclareMathOperator{\Mat}{Mat}

\newcommand{\Res}{\operatorname{Res}}
\newcommand{\GMP}{\text{\rm GMP}}

\newcommand{\tr}{\operatorname{tr}}

\newcommand{\id}{\mathrm{id}}

\newcommand{\diag}{\operatorname{diag}}

\newcommand{\CS}{\operatorname{CS}}

\newcommand{\cA}{{\mathcal{A}}}
\newcommand{\cB}{{\mathcal{B}}}

\newcommand{\cF}{{\mathcal{F}}}
\newcommand{\cG}{{\mathcal{G}}}

\newcommand{\cJ}{{\mathcal{J}}}
\newcommand{\cI}{{\mathcal{I}}}

\newcommand{\cL}{{\mathcal{L}}}

\newcommand{\cO}{{\mathcal{O}}}
\newcommand{\cP}{{\mathcal{P}}}

\newcommand{\cS}{{\mathcal{S}}}
\newcommand{\cT}{{\mathcal{T}}}

\newcommand{\fA}{{\mathfrak{A}}}

\newcommand{\fa}{{\mathfrak{a}}}

\newcommand{\fj}{{\mathfrak{j}}}

\newcommand{\fv}{{\mathfrak{v}}}
\newcommand{\fw}{{\mathfrak{w}}}

\newcommand{\bb}{\mathbf{b}}
\newcommand{\ba}{\mathbf{a}}

\newcommand{\bc}{\mathbf{c}}
\newcommand{\bo}{\mathbf{o}}
\newcommand{\bC}{\mathbf{C}}

\newcommand{\bp}{{\mathbf{p}}}

\newcommand{\bJ}{\mathbf{J}}
\newcommand{\bF}{\mathbf{F}}

\newcommand{\vare}{{\varepsilon}}
\newcommand{\e}{{\epsilon}}

\newcommand{\g}{\gamma}

\renewcommand{\l}{\lambda}
\renewcommand{\a}{\alpha}

\renewcommand{\b}{\beta}
\renewcommand{\d}{\delta}

\renewcommand{\L}{{\Lambda}}

\newcommand{\dd}{\mathrm{d}}

\renewcommand{\Re}{\operatorname{Re}}
\renewcommand{\Im}{\operatorname{Im}}
\renewcommand{\Cap}{\operatorname{Cap}}
\newcommand{\dist}{\operatorname{dist}}
\newcommand{\esssupp}{\operatorname{ess\,supp}}

\newcommand{\PSL}{\mathrm{PSL}}

\newcommand{\wlim}{\operatorname*{w-lim}}

\newcommand{\vp}{{\vec{p}}}

\newcommand{\vq}{{\vec{q}}}
\newcommand{\vbp}{{\vec{\mathbf{p}}}}

\newcommand{\ess}{\text{\rm{ess}}}

\newcommand{\bbD}{\mathbb{D}}
\newcommand{\bbN}{\mathbb{N}}
\newcommand{\bbZ}{\mathbb{Z}}

\newcommand{\bbR}{\mathbb{R}}
\newcommand{\bbC}{\mathbb{C}}

\renewcommand{\GMP}{\mathrm{GMP}}

\newcommand{\ac}{\mathrm{ac}}

\newcommand{\spann}{\mathrm{span}}

\allowdisplaybreaks

\title[Orthogonal rational functions with real poles, root asymptotics, GMP matrices]{Orthogonal rational functions with real poles, root asymptotics, and GMP matrices}

\thanks{B.E.\ was supported by Austrian Science Fund FWF, project no: J 4138-N32.}
\thanks{M.L.\ was supported in part by NSF grant DMS--1700179.}
\thanks{G.Y.\ was supported in part by NSF grant DMS--1745670.}

\author{Benjamin Eichinger, Milivoje Luki\'c, Giorgio Young}

\address{Institute of Analysis, Johannes Kepler University of Linz, 4040 Linz, Austria.}
\email{benjamin.eichinger@jku.at}
\address{Department of Mathematics, Rice University MS-136, Box 1892,
	Houston, TX 77251-1892, USA.}
\email{milivoje.lukic@rice.edu}
\address{Department of Mathematics, Rice University MS-136, Box 1892,
	Houston, TX 77251-1892, USA.}
\email{gfy1@rice.edu}

\begin{document}
\begin{abstract}
There is a vast theory of the asymptotic behavior of orthogonal polynomials with respect to a measure on $\mathbb{R}$ and its applications to Jacobi matrices. That theory has an obvious affine invariance and a very special role for $\infty$. We extend aspects of this theory in the setting of rational functions with poles on $\overline{\mathbb{R}} = \mathbb{R} \cup \{\infty\}$, obtaining a formulation which allows multiple poles and proving an invariance with respect to $\overline{\mathbb{R}}$-preserving M\"obius transformations. We obtain a characterization of Stahl--Totik regularity of a GMP matrix in terms of its matrix elements; as an application, we give a proof of a conjecture of Simon -- a Ces\`aro--Nevai property of regular Jacobi matrices on finite gap sets.
\end{abstract}

\maketitle

\section{Introduction}

There is a vast theory of orthogonal polynomials with respect to measures on $\bbC$ and their root asymptotics, exemplified by the Ullman--Stahl--Totik theory of regularity. Let $\mu$ be a compactly supported probability measure and $\{p_n\}_{n=0}^\infty$ the corresponding orthonormal polynomials, obtained by the Gram--Schmidt process from $\{z^n\}_{n=0}^\infty$ in $L^2(d\mu)$. Then
\begin{equation}\label{6aug1}
\liminf_{n\to\infty}  \lvert p_n(z) \rvert^{1/n} \ge e^{G_\E(z,\infty)}
\end{equation}
for $z$ outside the convex hull of $\supp \mu$, where $\E$ is the essential support of $\mu$ and $G_\E$ denotes the potential theoretic Green function for the domain $\overline{\bbC} \setminus \E$; if that domain is not Greenian, one takes $G_\E = +\infty$ instead. For measures compactly supported in $\bbR$, this theory can be interpreted in terms of self-adjoint operators. In particular, for any bounded  half-line Jacobi matrix
\[
J = \begin{pmatrix}
b_1 & a_1 &  &  \\
a_1 & b_2 & a_2 & \\
 & a_2 & \ddots & \ddots \\
 &  & \ddots &
 \end{pmatrix}
\]
with $a_\ell > 0$, $b_\ell \in \bbR$,
\begin{equation}\label{6aug2}
\limsup_{n\to\infty} \left( \prod_{\ell=1}^n a_\ell \right)^{1/n} \le \Cap \sigma_\ess(J),
\end{equation}
where $\Cap$ denotes logarithmic capacity. For both of these universal inequalities, the case of equality (and existence of limit) is called Stahl--Totik regularity \cite{StahlTotik92}; the theory originated with the case $\E = [-2,2]$, first studied by Ullman~\cite{Ullman}.

We extend aspects of this theory to the setting of rational functions with poles in $\overline{\bbR} = \bbR \cup \{\infty\}$. One motivation for this  is the search for a more conformally invariant theory. Statements such as \eqref{6aug1}, \eqref{6aug2} rescale in obvious ways with respect to affine transformations (automorphisms of $\bbC$) which preserve $\bbR$, so it is obvious that an affine pushforward of a Stahl--Totik regular measure is Stahl--Totik regular. However, the point $\infty$ has a very special role throughout the theory: for a M\"obius transformation $f$ which does not preserve $\infty$, $p_n \circ f$ are rational functions with a pole at $f^{-1}(\infty)$, and $f(J)$ as defined by the functional calculus is not a finite band matrix. Thus, it is a nontrivial question whether a M\"obius pushforward of a Stahl--Totik regular measure is Stahl--Totik regular.

The set of M\"obius transformations which preserve  $\overline{\bbR}$ is the semidirect group product $\PSL(2,\bbR) \rtimes \{ \id, z\mapsto -z \}$,  whose normal subgroup $\PSL(2,\bbR)$ corresponds to the orientation preserving case. Denote by $f_* \mu$ the pushforward of $\mu$, defined by $(f_* \mu)(A) = \mu(f^{-1}(A))$ for Borel sets $A$. As an example of our techniques, we obtain the following:

\begin{theorem} \label{thmST2}
Let $f \in  \PSL(2,\bbR) \rtimes \{ \id, z\mapsto -z \}$. If $\mu$ is a Stahl--Totik regular measure on $\bbR$ and $\infty \notin \supp  (f_* \mu)$, then the pushforward measure $f_*\mu$ is also Stahl--Totik regular.
\end{theorem}

However, we will mostly work in the more general setting when multiple poles on $\overline{\bbR}$ are allowed, which arises naturally in the spectral theory of self-adjoint operators. Denote by $T_{f,d\mu}$ the multiplication operator by $f$ in $L^2(d\mu)$. The matrix representation for $T_{x,d\mu(x)}$ in the basis of orthogonal polynomials is a Jacobi matrix, and through this classical connection, the theory of orthogonal polynomials is inextricably linked to the spectral theory of Jacobi matrices. In this matrix representation, resolvents $T_{(\bc-x)^{-1},d\mu(x)}$ are not finite-diagonal matrices. However, in a basis of orthogonal rational functions with poles at $\bc_1, \dots, \bc_g, \infty$, the multiplication operators $T_{(\bc_1-x)^{-1},d\mu(x)}$, \dots, $T_{(\bc_g-x)^{-1},d\mu(x)}$, $T_{x,d\mu(x)}$ all have precisely $2g+1$ nontrivial diagonals. The corresponding matrix representations are called GMP matrices; they were introduced by Yuditskii \cite{YuditskiiAdv}.

Self-adjoint operators and their matrix representations are an important part of this work, so we choose to present the theory in a more self-contained way, using self-adjoint operators from the ground up; this has similarities with \cite{Simon07}. Some proofs could be shortened by using orthogonal polynomials with respect to varying weights \cite[Chapter 3]{StahlTotik92}, but some facts rely on the precise structure obtained by the periodically repeating sequence of poles.

We should also compare this to the case of CMV matrices: for a measure supported on the unit circle, Stahl--Totik regularity is still defined in terms of orthogonal polynomials, but the CMV basis \cite{CMV03,Simon07} is given in terms of positive and negative powers of $z$, i.e., orthonormal rational functions with poles at $\infty$ and $0$. The symmetries in that setting lead to explicit formulas for the CMV basis in terms of the orthogonal polynomials; it is then a matter of calculation to relate the exponential growth rate of the CMV basis to that of the orthogonal polynomials, and to interpret regularity in terms of the CMV basis.  In our setting, there is no such symmetry and no formula for orthonormal rational functions in terms of orthonormal polynomials.

In order to state our results in a conformally invariant way, we will use the following notations and conventions throughout the paper. The measure $\mu$ will be a probability measure on $\overline{\bbR}$. We denote by $\supp \mu$ its support in $\overline{\bbR}$, and we consider its essential support (the support with isolated points removed), denoted
\[
\E = \esssupp \mu.
\]
We will always assume that $\mu$ is nontrivial; equivalently, $\E \neq \emptyset$.

Fix a finite sequence with no repetitions, $\bC = ( \bc_1,\dots, \bc_{g+1})$ with $\bc_k \in \overline{\bbR} \setminus \supp \mu$ for all $k$. Consider the sequence
$\{ r_n \}_{n=0}^\infty$ where $r_0 = 1$ and for $n = j (g+1) + k$, $1\le k \le g+1$,
\begin{equation}\label{rndefn}
r_n(z) = \begin{cases} \frac 1{( \bc_k - z)^{j+1}} & \bc_k \in \bbR \\
z^{j+1} & \bc_k = \infty
\end{cases}
\end{equation}
Applying the Gram--Schmidt process to this sequence in $L^2(d\mu)$ gives the sequence of orthonormal rational functions $\{ \tau_n \}_{n=0}^\infty$ whose behavior we will study. We note that the special case $\supp \mu \subset \bbR$, $g=0$, $\bC = (\infty)$ corresponds to the standard construction of orthonormal polynomials associated to the measure $\mu$ (note that, since we denote by $\supp \mu$ the support in $\overline{\bbR}$, the statement $\supp \mu \subset \bbR$ implies that $\mu$ is compactly supported in $\bbR$), and our first results are an extension of the same techniques.

The first result is a universal lower bound on the growth of $\{ \tau_n \}_{n=0}^\infty$ in terms of a potential theoretic quantity. If $\E$ is not a polar set, we use the (potential theoretic) Green function for the domain $\overline{\bbC} \setminus \E$, denoted $G_\E$, and we define
\begin{equation}\label{Glincomb}
\cG_\E(z,\bC) = \begin{cases}
\frac{1}{g+1} \sum_{k=1}^{g+1} G_\E(z,\bc_k) & \E\text{ is not polar} \\
+\infty & \E\text{ is polar}
\end{cases}
\end{equation}

\begin{theorem} \label{thm11}
For all $z \in \overline{\bbC} \setminus \overline{\bbR}$,
\[
\liminf_{n\to\infty}  \lvert \tau_{n}(z) \rvert^{1/n} \ge e^{\cG_\E(z,\bC)}.
\]
\end{theorem}

This is a good place to point out that our current setup is not related to the recent paper \cite{EichLuk}, in which the behavior was compared to a Martin function at a boundary point of the domain. Here, the behavior is compared to a combination of Green functions \eqref{Glincomb}, all the poles are in the interior of the domain $\overline{\bbC} \setminus \E$, and the difficulty comes instead from the multiple poles.

Another universal inequality for orthonormal polynomials comes from comparing their leading coefficients to the capacity of $\E$. In our setting, the analog of the leading coefficient must be considered in a pole-dependent way. Denote
\[
\cL_n = \spann \{ r_\ell \mid 0 \le \ell \le n\}.
\]
By the nature of the Gram--Schmidt process, there is a $\kappa_n > 0$ such that
\[
\tau_n - \kappa_n r_n \in \cL_{n-1}.
\]
The Gram--Schmidt process can be reformulated as the $L^2(d\mu)$-extremal problem
\begin{equation} \label{eq:L2extremal}
\kappa_n = \max\left\{\Re\kappa: f=\kappa r_n +h, h\in\cL_{n-1}, \|f\|_{L^2(d\mu)}\leq 1\right\}.
\end{equation}
By strict convexity of the $L^2$-norm, these $L^2$-extremal problems have unique extremizers given by $f = \tau_n$, and $\kappa_n$ is explicitly characterized as a kind of leading coefficient for $\tau_n$ with respect to the pole at $\bc_k$ where $n = j (g+1) + k$, $1 \le k \le g+1$. Below, we will also relate the constants $\kappa_n$ to off-diagonal coefficients of certain matrix representations.

The growth of the leading coefficients $\kappa_n$ will be studied along sequences $n=j(g+1)+k$ for a fixed $k$, and bounded by quantities related to the pole $\bc_k$. If $\E$ is not a polar set, it is a basic property of the Green function that the limits
\[
\gamma_\E^k=\begin{cases}\lim_{z\to\bc_k}(G_\E(z,\bc_k)+\log|z-\bc_k|),&\bc_k\ne \infty\\ \lim_{z\to\bc_k}(G_\E(z,\bc_k)-\log|z|),& \bc_k=\infty\end{cases}
\]
exist. Note that, if $\bc_k = \infty$, $\gamma_\E^k$ is precisely the Robin constant for the set $\E$. We further define constants $\lambda_k$ by
\begin{equation}\label{6aug5}
\log \lambda_k = \begin{cases} \g_\E^k  + \sum_{\substack{1\leq \ell \leq g+1\\ \ell \neq k}}G_\E(\bc_k,\bc_\ell) & \E\text{ is not polar} \\
+\infty & \E\text{ is polar}
\end{cases}
\end{equation}

\begin{theorem} \label{thm12}
For all $1 \le k \le g+1$, for the subsequence $n(j)=j(g+1)+k$,
\begin{equation}\label{16aug1}
\liminf_{j \to \infty} \kappa_{n(j)}^{1/n(j)} \geq \lambda_k^{1/ (g+1)}.
\end{equation}
\end{theorem}

\begin{theorem}\label{thm:CRegMain}
The following are equivalent:
	\begin{enumerate}[(i)]
		\item\label{it:CRegConstants3} For some $1 \le k \le g+1$,   for the subsequence $n(j)=j(g+1)+k$,
		\[
		\lim_{j\to\infty} \kappa_{n(j)}^{1/n(j)} =  \lambda_k^{1 / (g+1)};
		\]
		\item\label{it:CRegConstants} For all $1 \le k \le g+1$,   for the subsequence $n(j)=j(g+1)+k$,
		\[
		\lim_{j\to\infty} \kappa_{n(j)}^{1/n(j)} =  \lambda_k^{1 / (g+1)};
		\]
		\item\label{it:CRegConstants2}
		\[\lim_{n\to\infty} \left( \prod_{\ell=1}^{g+1} \kappa_{n+\ell} \right)^{1/n} = \left( \prod_{k=1}^{g+1} \lambda_k \right)^{1/(g+1)}
		\]
\item\label{it:CRegBoundary} For q.e. $z\in\E$, we have
$\limsup_{n\to\infty} |\tau_{n}(z)|^{1/n} \leq 1$;
\item \label{it:upperhalfplane} For some $z\in\C_+$, $\limsup_{n\to\infty} |\tau_{n}(z)|^{1/n} \leq e^{\cG_\E(z,\bC)}$;
\item\label{it:plane}	For all $z\in\C$,
$\limsup_{n\to\infty}  |\tau_{n}(z)|^{1/n} \leq e^{\cG_\E(z,\bC)}$;
\item  \label{it:uniformlimit} Uniformly on compact subsets of $\C\setminus\R$,
$\lim_{n\to\infty} |\tau_{n}(z)|^{1/n}= e^{\cG_\E(z,\bC)}$.
\end{enumerate}
\end{theorem}

\begin{definition}
The measure $\mu$ is $\bC$-regular if it obeys one (and therefore all) of the assumptions of Theorem~\ref{thm:CRegMain}.
\end{definition}

In this terminology, Stahl--Totik regularity is precisely $(\infty)$-regularity, i.e., $\bC$-regularity for the special case $\supp \mu \subset \bbR$, $g=0$, $\bC = (\infty)$. Theorems~\ref{thm11}, \ref{thm12}, \ref{thm:CRegMain} are closely motivated by foundational results for Stahl--Totik regularity. A new phenomenon appears through the periodicity with which poles are taken in \eqref{rndefn} and the resulting subsequences $n(j) = j(g+1)+k$: since $\kappa_n$ is a normalization constant for $\tau_n$, it is notable that control of $\kappa_n$ along a single subsequence $n(j) = j(g+1)+k$ in Theorem~\ref{thm:CRegMain}.\eqref{it:CRegConstants3}  provides control over the entire sequence. This phenomenon doesn't have an exact analog for orthogonal polynomials, where $g=0$. We will also see below that this is essential in order to characterize the regularity of a GMP matrix using only the entries of the matrix itself and not its resolvents.

Moreover, we show that the regular behavior described by Theorem~\ref{thm:CRegMain} is independent of the set of poles $\bC$:

\begin{theorem} \label{thmST0}
Let $\bC_1, \bC_2$ be two finite sequences of elements from $\overline{\bbR} \setminus \supp \mu$, not necessarily of the same length. Then $\mu$ is $\bC_1$-regular if and only if it is $\bC_2$-regular.
\end{theorem}

\begin{corollary} \label{thmST1}
Let $\supp \mu \subset \bbR$. Let $\bC$ be a finite sequence of elements from $\overline{\bbR} \setminus \supp \mu$. Then $\mu$ is $\bC$-regular if and only if it is Stahl--Totik regular.
\end{corollary}

Thus, Theorem~\ref{thm:CRegMain} should not be seen as describing equivalent conditions for a new class of measures, but rather a new set of regular behaviors for the familiar class of Stahl--Totik regular measures.

We consistently work with poles on $\overline{\bbR}$ since our main interest is tied to self-adjoint problems. Some of our results are in a sense complementary to the setting of \cite[Section 6.1]{StahlTotik92}, where poles are allowed in the complement of the convex hull of $\supp \mu$, and the behavior of orthogonal rational functions is considered with respect to a Stahl--Totik regular measure. Due to this, it is natural to expect that these results hold more generally, for measures on $\bbC$ and general collections of poles and M\"obius transformations. Moreover, in our setup the poles are repeated exactly periodically, but we expect this can be generalized to a sequence of poles which has a limiting average distribution. Related questions for orthogonal rational functions were also studied by \cite{BuGoHeNjORFBook,DeckLub12}.

As noted in \cite[Section 6.1]{StahlTotik92}, poles in the gaps of $\supp \mu$ can cause interpolation defects in the problem of interpolation by rational functions. In our work, these interpolation defects show up as possible reductions in the order of the poles.  For example, consider $\bC = (\infty,0)$. Then, by construction, $\tau_{2j+1}$ is allowed a pole at $0$ of order at most $j$. However, if $\mu$ is symmetric with respect to $z\mapsto -z$, the functions $\tau_n$ will have an even/odd symmetry. Since $\tau_{2j+1}$ contains a nontrivial multiple of $z^{j+1}$, it follows that $\tau_{2j+1}(z) = (-1)^{j+1} \tau_{2j+1}(-z)$.  By this symmetry, the actual order of the pole at $0$ is $j+1-k$ for some even $k$, so it cannot be equal to $j$ (it will follow from our results that in this case, the order of the pole is $j-1$). The same effect can be seen for the pole at $\infty$ for $\bC = (0,\infty)$. In the polynomial case, this does not occur: $p_n$ always has a pole at $\infty$ of order exactly $n$.

We will consider at once the distribution of zeros of $\tau_n$ and the possible reductions in the order of the poles. We will prove that all zeros of $\tau_n$ are real and simple, and that $n-g \le \deg \tau_n \le n$. We define the normalized zero counting measure
\[
\nu_n = \frac 1n \sum_{w: \tau_n(w)=0}  \delta_w.
\]
Although we normalize by $n$, $\nu_n$ may not be a probability measure: however $1 - g/n \le \nu_n(\overline{\bbR}) \le 1$. Therefore,  normalizing  by $\deg \tau_n$ instead of by $n$ would not affect the limits as $n\to\infty$.

We will now describe the weak limit behavior of the measures $\nu_n$ as $n\to\infty$. To avoid pathological cases, we assume that $\E$ is not polar; in that case, denoting   by $\omega_\E( dx,w)$ the harmonic measure for the domain $\overline{\bbC} \setminus \E$ at the point $w$, we define the probability measure on $\E$,
\[
\rho_{\E,\bC} = \frac 1{g+1} \sum_{j=1}^{g+1} \omega_\E(\dd x,\bc_j).
\]
The results below describe weak limits of measures in the topology dual to $C(\overline{\bbR})$.

\begin{theorem} Let $\mu$ be a probability measure on $\overline{\bbR}$. Assume that $\E$ is not a polar set.  \label{thmDOSlimit}
\begin{enumerate}[(a)]
\item If $\mu$ is $\bC$ regular, then $\wlim_{n\to\infty} \nu_n = \rho_{\E,\bC}$.
\item If $\wlim_{n\to\infty} \nu_n =\rho_{\E,\bC}$, then $\mu$ is $\bC$ regular or there exists a polar set $X \subset \E$ such that $\mu(\overline{\bbR} \setminus X) = 0$.
\end{enumerate}
\end{theorem}

We now turn to matrix representations of self-adjoint operators. Fix a sequence $\bC = (\bc_1,\dots,\bc_{g+1})$ such that $\bc_{k_\infty} = \infty$ for some {$1\leq k_\infty\leq g+1$}. A half-line GMP matrix \cite{YuditskiiAdv} is the matrix representation for multiplication by $x$ in the basis $\{\tau_n \}_{n=0}^\infty$ {for this sequence $\bC$}; its matrix elements are
\[
A_{mn} = \int \overline{\tau_m(x)} x \tau_n(x) \,d\mu(x).
\]
The condition that $\bc_{k_\infty}  = \infty$ for some $k_\infty$ guarantees that $A_{mn} = 0$ for $\lvert m - n \rvert  > g+1$, so these matrix elements generate a bounded operator $A$ on $\ell^2(\bbN_0)$ such that $A_{mn} = \langle e_m,  A e_n \rangle$, where $(e_n)_{n=0}^\infty$ denotes the standard basis of $\ell^2(\bbN_0)$. We say that $A \in \A(\bC)$.

GMP matrices have the property that some of their resolvents are also GMP matrices; namely, for any $k \neq k_\infty$, $(\bc_k - A)^{-1} \in \A(f(\bC))$ where $f$ is the M\"obius transform $f: z \mapsto (\bc_k - z)^{-1}$ and $f(\bC) = (f(\bc_1), \dots, f(\bc_{g+1}))$.

Note that the special case $g=0$, $\bC = (\infty)$ gives precisely a Jacobi matrix. A Jacobi matrix is said to be regular if it is obtained by this construction from a regular measure; analogously, we will call a GMP matrix regular if it is obtained from a regular measure. Just as  regularity of a Jacobi matrix can be characterized in terms of its off-diagonal entries, we will show that regularity of a GMP matrix can be characterized in terms of its entries in the outermost nontrivial diagonal. We will also obtain a GMP matrix analog of the inequality \eqref{6aug2}.

The GMP matrix has an additional block matrix structure; in particular, for a GMP matrix with $\bc_{k_\infty} = \infty$, on the outermost nonzero diagonal $m=n - g-1$, the only nonzero terms appear for $n = j(g+1)+k_\infty$, and those are strictly positive. Thus, we denote
\begin{equation} \label{betaDefn}
\beta_j =  \langle e_{j(g+1)+k_\infty}, Ae_{(j+1)(g+1)+k_\infty} \rangle
\end{equation}

\begin{theorem} \label{thmRegularityGMP}
Fix a probability measure $\mu$ with $\supp \mu \subset \bbR$ and a sequence $\bC = (\bc_1,\dots, \bc_{g+1})$ with $\bc_k = \infty$. Then
\begin{equation}\label{16aug3}
\limsup_{j \to\infty} \left( \prod_{\ell=1}^j \beta_\ell \right)^{1/ j } \le \lambda_{k_\infty}^{-1}.
\end{equation}

Moreover, the measure $\mu$ is Stahl--Totik regular if and only if
\begin{equation}\label{16aug4}
\lim_{j \to\infty} \left( \prod_{\ell=1}^j \beta_\ell \right)^{1/ j } =  \lambda_{k_\infty}^{-1}.
\end{equation}
\end{theorem}

The proof will use a relation between the sequence $\{\beta_j\}_{j=1}^\infty$ and the constants $\{\kappa_{j(g+1)+k_\infty}\}_{j=1}^\infty$. 
In particular, the characterization of regularity in Theorem~\ref{thmRegularityGMP} is made possible by the characterization of regularity in terms of the subsequence $\{ \kappa_{j(g+1)+k}\}_{j=1}^\infty$ for any single $k$. Theorem~\ref{thmRegularityGMP} also corroborates the perspective that regularity of the measure is the fundamental notion which manifests itself equally well in many different matrix representations.

Since the resolvents {$(\bc_k - A)^{-1}$} are also GMP matrices and their measures are pushforwards of the original measure, they are also regular GMP matrices; in this sense, Theorem~\ref{thmRegularityGMP} provides $g+1$ criteria for regularity, one corresponding to each subsequence $n(j) = j(g+1)+{k}$, $1\leq k\leq g+1$.

As an application of this theory, we show that it provides a proof of a theorem for Jacobi matrices originally conjectured by Simon \cite{SimonJAT09}. Let  $\E \subset \bbR$ be a compact finite gap set,
\begin{equation}\label{finitegapset}
\E=[\bb_0,\ba_0]\setminus\bigcup_{k=1}^g (\ba_k,\bb_k),
\end{equation}
and denote by $\cT_\E^+$ the set of almost periodic half-line Jacobi matrices with $\sigma_\ess(J) = \sigma_\ac(J) = \E$ \cite{ChriSiZiFiniteGap,GeHoMiTe}. Through algebro-geometric techniques and the reflectionless property, this class of Jacobi matrices has been widely studied for their spectral properties and quasiperiodicity. They also provide natural reference points for perturbations, which is our current interest. On bounded half-line Jacobi matrices $J$, we consider the metric
\begin{equation}\label{18aug1}
d(J, \tilde J) = \sum_{k=1}^\infty e^{-k} ( \lvert a_k - \tilde a_k \rvert + \lvert b_k - \tilde b_k \rvert).
\end{equation}
On norm-bounded sets of Jacobi matrices, convergence in this metric corresponds to strong operator convergence. However, instead of distance to a fixed Jacobi matrix $\tilde J$, we will consider the distance to $\cT_\E^+$,
\[
d(J, \cT_\E^+) = \inf_{\tilde J \in \cT_\E^+} d(J, \tilde J)  = \min_{\tilde J \in \cT_\E^+} d(J, \tilde J).
\]
Denote by $S_+$ the right shift operator on $\ell^2(\bbN_0)$, $S_+ e_n = e_{n+1}$. The condition  $d( (S_+^*)^m J S_+^m, \cT_\E^+)  \to 0$ as $m\to\infty$ is called the Nevai condition. For $\E = [-2,2]$, this corresponds simply to the commonly considered condition $a_n \to 1$, $b_n \to 0$ as $n\to\infty$  \cite{Nevai}. In general, as a consequence of \cite{Remling11}, the Nevai condition implies regularity. The converse is false; however:

\begin{theorem}\label{conj1}
If $\E \subset \bbR$ is a compact finite gap set and $J$ is a regular Jacobi matrix with $\sigma_\ess(J) = \E$, then
\begin{equation}\label{24jun2}
\lim_{N\to\infty}\frac 1N \sum_{m=1}^N d( (S_+^*)^m J S_+^m, \cT_\E^+) = 0.
\end{equation}
\end{theorem}

The condition \eqref{24jun2} is described as the Ces\`aro--Nevai condition; it was first studied by Golinskii--Khrushchev \cite{GolKru} in the OPUC setting with essential spectrum equal to $\partial\bbD$. Theorem~\ref{conj1} was conjectured by Simon \cite{SimonJAT09}  and proved in the special case when $\E$ is the spectrum of a periodic Jacobi matrix with all gaps open by using the periodic discriminant and techniques from Damanik--Killip--Simon \cite{DamKillipSimAnnals} to reduce to a block Jacobi setting. It was then proved by Kr\"uger \cite{Krueger10} by very different methods under the additional assumption $\inf_{n} a_n > 0$. While this is a common assumption in the ergodic literature, regular Jacobi matrices do not always satisfy it: \cite[Example 1.4]{Simon07} can easily be modified to give a regular Jacobi matrix with spectrum $[-2,2]$ and $\inf a_n = 0$. We prove Theorem~\ref{conj1} in full generality by applying Simon's strategy and, instead of the periodic discriminant and techniques from \cite{DamKillipSimAnnals}, using the Ahlfors function, GMP matrices, and techniques of Yuditskii \cite{YuditskiiAdv}.

For the compact finite gap set $\E \subset \bbR$, among all analytic functions $\overline{\bbC} \setminus \E \to \bbD$ which vanish at $\infty$, the Ahlfors function $\Psi$ takes the largest value of $\Re (z \Psi(z))\rvert_{z=\infty}$. The Ahlfors function has precisely one zero in each gap, denoted $\bc_k \in (\ba_k, \bb_k)$ for $1 \le k \le g$, a zero at $\bc_{g+1} = \infty$, and no other zeros; see also \cite[Chapter 8]{SimonBasicCompAna}. In particular, for the finite gap set $\E$, this generates a particularly natural sequence of poles $\bC_\E = (\bc_1, \dots, \bc_g, \infty)$.

The Ahlfors function was used by Yuditskii \cite{YuditskiiAdv} to define a discriminant for finite gap sets,
\begin{equation} \label{13aug1}
\Delta_\E(z)=\Psi(z)+\frac{1}{\Psi(z)}.
\end{equation}
This function is not equal to the periodic discriminant, but it has some similar properties and it is available more generally (even when $\E$ is not a periodic spectrum). Namely, $\Delta_\E$ extends to a meromorphic function on $\overline{\bbC}$ and $(\Delta_\E)^{-1}([-2,2]) = \E$. It was introduced by Yuditskii to solve the Killip--Simon problem for finite gap essential spectra. In fact, the discriminant is a rational function of the form
\begin{align}\label{eq:discrRational}
	\Delta_\E(z)  =\l_{g+1} z + d + \sum_{k=1}^g\frac{\l_k}{\bc_k-z}
\end{align}
for some $d \in \bbR$; in particular, we will explain that the constants $\lambda_j > 0$ in \eqref{eq:discrRational} match the general definition \eqref{6aug5}.

As a first glimpse of our proof of Theorem~\ref{conj1}, we note that it uses the following chain of implications. Starting with a regular Jacobi matrix with essential spectrum $\E$, by a change of one Jacobi coefficient, which does not affect regularity, we can assume that $\bc_k \notin \supp \mu$ (Lemma~\ref{lemmarank1}). Under this assumption, regularity of the Jacobi matrix implies regularity of the corresponding GMP matrix $A$ and the resolvents $(\bc_k - A)^{-1}$, $k=1,\dots, g$, which can be characterized in terms of their coefficients by Theorem~\ref{thmRegularityGMP}. By properties of the Yuditskii discriminant, this further implies regularity of the block Jacobi matrix $\Delta_\E(A)$. Let us briefly recall that a block Jacobi matrix is of the form
\begin{equation}\label{16aug7}
\bJ = \begin{bmatrix}
\fw_0 & \fv_0 & & & & \\
\fv_0^* & \fw_{1}& \fv_1 & &&\\
 & \fv_1^* & \fw_2 & \fv_2 & &\\
&  & \fv_2^* & \ddots&\ddots& \\
& &  & \ddots & &
\end{bmatrix}
\end{equation}
where $\fv_j$ and $\fw_j$ are $d \times d$ matrices, $\fw_j = \fw_j^*$,  and $\det \fv_j \ne 0$ for each $j$. Type 3 block Jacobi matrices have each $\fv_j$ lower triangular and positive on the diagonal. An extension of regularity to block Jacobi matrices was developed by Damanik--Pushnitski--Simon \cite{DamPushnSim08}; in particular, $\bJ$ is regular for the set $[-2,2]$ if $\sigma_\ess(\bJ) = [-2,2]$ and
\begin{equation}\label{13aug5}
\lim_{n\to \infty}\left(\prod_{j =1}^n \lvert \det \fv_j \rvert \right)^{1/n}=1.
\end{equation}

This chain of arguments will result in the following lemma:

\begin{lemma}\label{lem:BlockJacobiRegular}
Let $J$ be a regular Jacobi matrix, $\E = \sigma_\ess(J)$ a finite gap set, and $\bC_\E$ the corresponding sequence of zeros of the Ahlfors function. Assuming $\bc_k \notin \sigma(J)$ for $1\le k\le g$, denote by $A$ the GMP matrix corresponding to $J$ with respect to the sequence $\bC_\E$. Then $\Delta_\E(A)$ is a regular type 3 block Jacobi matrix with essential spectrum $[-2,2]$.
\end{lemma}

With this lemma, it will follow that $\bJ = \Delta_\E(A)$ obeys a Ces\`aro--Nevai condition. That Ces\`aro--Nevai condition will imply \eqref{24jun2} by a modification of arguments of \cite{YuditskiiAdv}. The strategy is clear: just as \cite{YuditskiiAdv} uses a certain square-summability in terms of $\fv_j, \fw_j$ to prove finiteness of $\ell^2$-norm of $\{ d((S_+^*)^m J S_+^m, \cT_\E^+) \}_{m=0}^\infty$, we will use Ces\`aro decay in terms of $\fv_j, \fw_j$ to conclude the Ces\`aro decay \eqref{24jun2}. This can be expected due to a certain locality in the dependence between the terms of the series considered; this idea first appeared in \cite{SimonJAT09} in the setting of periodic spectra with all gaps open. However, some care is needed, since the locality is only approximate in some steps; this is already visible in \eqref{18aug1}.  Also, substantial modifications are needed throughout the proof due to the possibility of $\liminf \lVert \fv_j \rVert = 0$ (this cannot happen in the Killip--Simon class), which locally breaks some of the estimates. The fix is that this can only happen along a sparse subsequence, but the combination of a bad sparse subsequence and approximate locality means that we cannot simply ignore a bad subsequence once from the start; we must maintain it throughout the proof. A related issue arises with the Ces\`aro version of a Killip--Simon type functional. We will describe the necessary modifications to the detailed analysis in \cite{YuditskiiAdv}.

The rest of the paper will not exactly follow the order given in this introduction. In Section 2, we describe the behavior of our problem with respect to M\"obius transformations, and we describe the distribution of zeros of the rational function $\tau_n$. In Section 3, we recall the structure of GMP matrices and relate their matrix coefficients to the quantities $\kappa_n$, and use this to provide a first statement about exponential growth of orthonormal rational functions on $\overline{\bbC} \setminus \overline{\bbR}$. In Section 4, we combine this with potential theoretic techniques to characterize limits of $\frac 1n \log \lvert \tau_n \rvert$ as $n \to \infty$ and prove the universal lower bounds. In Section 5, we prove the results for $\bC$-regularity and Stahl--Totik regularity. In Section 6 we describe a proof of Theorem~\ref{conj1}.

\section{Orthonormal rational functions and M\"obius transformations}

In the introduction, starting from the measure $\mu$ and sequence of poles $\bC$, we defined a sequence $\{r_n\}_{n=0}^\infty$ and the orthonormal rational functions  $\{ \tau_n \}_{n=0}^\infty$. In the next statement, we will denote these by $r_n(z; \bC)$ and $\tau_n(z;\mu,\bC)$, in order to state precisely the invariance of the setup with respect to M\"obius transformations.

\begin{lemma} \label{lemmaMobius}
If $f$ is a M\"obius transformation which preserves $\overline{\bbR}$, then
\begin{equation}\label{6aug3}
\tau_n(z;\mu,\bC) = \rho^n \tau_n(f(z);f_* \mu, f(\bC)),
\end{equation}
where $f(\bC) = (f(\bc_1), \dots, f(\bc_{g+1}))$ and
\[
\rho = \begin{cases} +1 & f \in \PSL(2,\bbR) \\
-1 & f \in \left( \PSL(2,\bbR) \rtimes \{ \id, z\mapsto -z \} \right) \setminus \PSL(2,\bbR)
\end{cases}
\]
\end{lemma}

\begin{proof}
Note that the sequence $\{r_n\}_{n=0}^\infty$ does not have this property: $r_n(z;\bC)$ is not equal to $\rho^n r_n(f(z);f(\bC))$. However, if we denote
\[
\cL_n(\bC) = \spann \{ r_\ell (\cdot; \bC) \mid 0 \le \ell \le n\},
\]
then it suffices to have
\begin{equation}\label{18jul1}
r_n(f(z);f(\bC)) - c_n \rho^n  r_n(z;\bC) \in \cL_{n-1}(\bC)
\end{equation}
for some constants $c_n > 0$. If \eqref{18jul1} holds, then applying the Gram--Schmidt process to the sequences $\{ r_n(f(z);f(\bC)) \}_{n=0}^\infty$ and $\{r_n(z;\bC)\}_{n=0}^\infty$ will give the same sequence of orthonormal functions, up to the sign change $\rho^n$, which is precisely \eqref{6aug3}.

Note that, if \eqref{6aug3} holds for $f_1, f_2$, it holds for their composition, so it suffices to verify \eqref{18jul1} for a set of generators of $\PSL(2,\bbR) \rtimes \{ \id, z\mapsto -z \}$.  In particular, \eqref{18jul1} is checked by straightforward calculations for affine transformations and for the inversion $f(z) = -1/z$, which implies the general statement since affine maps and inversion generate $\PSL(2,\bbR) \rtimes \{ \id, z\mapsto -z \}$.
\end{proof}

Let us emphasize what this lemma does and what it doesn't do. Since the M\"obius transformation acts on both the measure and the sequence of poles, Lemma~\ref{lemmaMobius}  does not by itself prove Theorem~\ref{thmST2}. Lemma~\ref{lemmaMobius} can only say that if $\mu$ is Stahl--Totik regular, then $f_* \mu$ is $(f(\infty))$-regular, which is not sufficient unless $f$ is affine. The proof of Theorem~\ref{thmST2} will be more involved.

However, Lemma~\ref{lemmaMobius} provides a very useful conformal invariance for many of our proofs. This can be compared to choosing a convenient reference frame. Since potential theoretic notions such as Green functions are conformally invariant, our results will be invariant with respect to M\"obius transformations. We will often use this invariance in the proofs to fix a convenient point at $\infty$.

Note that this will be possible even though some objects are not conformally invariant. Some of our results compare the sequences $\kappa_n$ with the $\lambda_k$, and although those objects are not preserved under conformal transformations, both sequences are affected in a compatible way so that the inequalities and equalities are preserved. Explicitly, fix $k$ and $n = j(g+1)+k$ and a M\"obius transformation $f \in \PSL(2,\bbR)$ (a reflection can be considered separately). Let us denote a local dilation factor $f'(\bc_k)= \lim_{z\to\bc_k} \frac{r_k(z,\bC)}{r_k(f(z),f(\bC))}$. Then, we use Lemma~\ref{lemmaMobius} to compute
\begin{align*}
\kappa_{n(j)}&=\lim_{z\to \bc_k}\frac{\tau_n(z,\mu,\bC)}{r_n(z,\bC)}=
\lim_{z\to \bc_k}\frac{\tau_n(f(z),f_*\mu,f(\bC))}{r_n(z,\bC)} =\frac{\tilde \kappa_{n(j)}}{f'(\bc_k)^{j+1}},
\end{align*}
where $\tilde \kappa_{n(j)}$ is the leading coefficient $\tau_n(z,f_*\mu,f(\bC))-\tilde\kappa_nr_n(z,f(\bc_k))\in \cL_{n-1}(f(\bC))$. If $\E$ is nonpolar, the Green function is conformally invariant so we find by another computation
\begin{align*}
\log\tilde\lambda_k&:=\lim_{w\to f(\bc_k)}\left( G_{f(\E)}(w,f(\bc_k)) -\log\lvert r_k(w,f(\bC)) \rvert \right)+\sum_{\substack{1\leq \ell\leq k\\\ell\ne k}}G_{f(\E)}(f(\bc_k),f(\bc_\ell))\\
&=\lim_{z\to \bc_k}\left( G_{\E}(z,\bc_k) - \log\lvert r_k(z,\bC) \rvert + \log \left\lvert \frac{r_k(z,\bC)}{r_k(f(z),f(\bC))} \right\rvert \right)+\sum_{\substack{1\leq \ell\leq k\\\ell\ne k}}G_{\E}(\bc_k,\bc_\ell)\\
&=\log\lambda_k+\log (f'(\bc_k)),
\end{align*}
where we have used that $f\in\PSL(2,\R)\implies f'>0$ on $\overline\R$. Thus, $\tilde \lambda_k=f'(\bc_k)\lambda_k$. If
$\E$ is polar, then $f(\E)$ is as well. From these calculations, it becomes elementary to verify that statements such as those in Theorems~\ref{thm12}, \ref{thm:CRegMain} are conformally invariant.

Note that technical ingredients of the proof, such as polynomial factorizations, give a preferred role to $\infty$ so they break symmetry. For instance, we will often use the observation that the subspace $\cL_n$ can be represented as
\begin{equation}\label{LnRndefn}
\cL_n = \left\{ \frac P{R_n} \mid P \in \cP_n\right\}
\end{equation}
for some suitable polynomial $R_n$ with factors which account for finite poles $\bc_k \neq \infty$. We will use the representation \eqref{LnRndefn} after placing a convenient point at $\infty$.  This idea is already seen in the next proof.

\begin{lemma}\label{lem:zeroesofnumeratornew}
 All zeros of the rational function $\tau_{n}$ are simple and lie in $\overline{\bbR}$. Moreover, $n-g \le \deg \tau_n \le n$.

Let $n=j(g+1) + k$, $1\le k \le g+1$, and denote by $I$ the connected component of $\bc_k$ in $\overline{\bbR} \setminus \supp \mu$. Then $\tau_{n}$ has no zeros in $I$ and at most one zero in any other connected component of $\overline{\bbR} \setminus \supp \mu$.
\end{lemma}

\begin{proof}
Fix $1\leq k\leq g+1$ and without loss of generality, assume $\bc_k = \infty$. Then, in the representations \eqref{LnRndefn}, we can notice that $R_{n-1} = R_n$. In particular, then $\tau_{n} \in \cL_n \setminus \cL_{n-1}$ implies the representation $\tau_{n(j)} = \frac{P_{n}}{R_n}$ for some polynomial $P_n$ of degree $n$.

Recall that $\tau_n$, $n=k+(j-1)(g+1)$ is the unique minimizer for the extremal problem \eqref{eq:L2extremal}. By complex conjugation symmetry, the minimizer is real. To proceed further, we study zeros of $P_n$  by using Markov correction terms.

We say that a rational function $M$ is an admissible Markov correction term if $M > 0$ a.e.\ on $\E$ and $M(z) P_n(z) \in \cP_{n-1}$. In this case, using $\langle M \tau_n, \tau_n \rangle > 0$, we see that the function $g(\epsilon) = \lVert \tau_n - \epsilon M \tau_n \rVert^2$ obeys
\[
g'(0) = - 2 \langle M \tau_n, \tau_n \rangle < 0.
\]
Thus, for small enough $\epsilon > 0$, the function
\[
\tilde \tau_n =  \tau_n - \epsilon M \tau_n
\]
obeys $\lVert \tilde\tau_n \rVert_{L^2(d\mu)}  < \lVert \tau_n\rVert_{L^2(d\mu)}$. Since $\tilde \tau_n$ is of the form $\tilde \tau_n=\kappa_nz^{j+1}+h(z)$ for some $h(z)\in \cL_{n-1}$ and in particular has the same leading coefficient as $\tau_n$, the function $\tilde\tau_n / \lVert \tilde\tau_n \rVert_{L^2(d\mu)} \in \cL_n$ contradicts extremality of $\tau_n$.  In other words, for the extremizer $\tau_n$, there cannot be any admissible Markov correction terms.

Assume that $P_n$ has a non-real zero $w\in \overline{\C}\setminus \overline{\R}$. Then, since $\tau_n$ is real, $P_n(\overline{w})=0$, so the Markov correction term $M(z;w)= \frac 1{(z-w)(z-\bar w)}$ would be admissible, leading to contradiction.

Assume that $P_n$ has two zeros  $x_1,x_2$ in the same connected component of
$\overline{\R}\setminus\supp\mu$; then, the Markov correction term
\[
M(z;x_1,x_2)=
\frac{1}{(z-x_1)(z-x_2)},
\]
would be admissible, leading to contradiction.

There are no zeros of $P_n$ in $I$. Otherwise, if $x \in I$ was a zero, the Markov term
\[
M(z,x) = \begin{cases}
\frac{1}{z-x},& x < \inf \E \\
\frac{1}{x-z}, & x > \sup \E
\end{cases}
\]
would be admissible.

Finally, all zeros of $P_n$ are simple: otherwise, if $x_0 \in \bbR$ was a double zero, the Markov term
\[
M(z,x_0) = \frac 1{(z-x_0)^2}
\]
would be admissible.

The properties of zeros of $\tau_n$ follow from those of $P_n$. There may be cancellations in the representation $\tau_n = \frac{P_n}{R_n}$, but since $P_n$ has at most a simple zero at $\bc_\ell$, the only possible cancellations are simple factors $(z-\bc_\ell)$, $\ell \neq k$. Thus, $n- g \le \deg \tau_n \le n$.
\end{proof}

The use of Markov correction factors is standard in the Chebyshev polynomial literature and is applied here with a modification for the $L^2$-extremal problem (in the $L^\infty$-setting, singularities in $M$ are treated with a separate argument near the singularity, which would not work here).

\begin{corollary}\label{cor:precompMeasure}
The measures $\nu_n$ are a precompact family with respect to weak convergence on $C(\overline{\bbR})$. Any accummulation point $\nu = \lim_{\ell\to\infty} \nu_{n_\ell}$ is a probability measure and $\supp\nu \subset \E$.
\end{corollary}
\begin{proof}
By Lemma~\ref{lem:zeroesofnumeratornew}, $\nu_n(\overline{\R})\leq 1$, so precompactness follows by the Banach-Alaoglu theorem. If
$\nu = \lim_{\ell\to\infty} \nu_{n_\ell}$, then since $1-\frac{g}{n_\ell}\leq \nu_{n_{\ell}}(\overline{\R})\leq 1$, $\nu(\overline{\R})=1$.

Let $(\ba,\bb)$ be a connected component of $\overline{\bbR} \setminus \E$. Let us prove that $\nu((\ba,\bb)) = 0$. By M\"obius invariance, it suffices to assume that $(a,b)$ is a bounded subset of $\bbR$.

Fix $r\in \N$. As $\supp \mu \setminus \E$ is a discrete set,
we have
\begin{align*}
  \#\{ x\in \supp \mu: \ba+1/r< x< \bb-1/r \}=M<\infty.
\end{align*}
So, by Lemma~\ref{lem:zeroesofnumeratornew}, $\nu_{n_\ell}((\ba+1/r, \bb-1/r))\leq \frac{2M+1}{n_\ell}$ and by the Portmanteau theorem and sending $r\to \infty$, $\nu((\ba,\bb))=0$ and
$\supp \nu \subset \E$.
\end{proof}

\section{GMP matrices and exponential growth of orthonormal rational functions}\label{sec:GMPandGrowth}

In this section, we consider orthonormal rational functions through the framework of GMP matrices. We begin by recalling the structure of GMP matrices \cite{YuditskiiAdv}. The GMP matrix has a tridiagonal block matrix structure, with the beginnings of new blocks corresponding to occurrences of $\bc_{k_\infty} = \infty$. Explicitly,
\begin{align*}
A=\begin{bmatrix}
\tilde B_{0} & \tilde A_{0} & & & & \\
\tilde A_{0}^* & B_{1}& A_{1}& &&\\
 & A_{1}^*& B_{2}&A_{2}& &\\
&  & A_2^* & \ddots&\ddots& \\
& &  & \ddots & &
\end{bmatrix}
\end{align*}
where $\tilde B_0$ is a $k_\infty\times k_\infty$ matrix, $\tilde A_0$ is a $k_\infty \times (g+1)$ matrix. For $j\ge 0$, $A_j, B_j$ are $(g+1) \times (g+1)$ matrices; {while for $j\geq 1$ these appear in $A$ unmodified in the above, $\tilde A_0$ and $\tilde B_0$ are projections of $A_0$ and $B_0$ respectively}. More precisely, let $X^-$ denote the upper triangular part of a matrix $X$ (excluding the diagonal) and $X^+$ the lower triangular part (including the diagonal). Then, indexing the entries of $A_j$, $B_j$, $j\geq 0$ from $0$ to $g$, we see they are of the form
\begin{align}\label{eq:structureGMP}
A_j=\vp_j\vec\delta_{0}^{\,\intercal},\quad B_j=\hat{\mathbf{C}}+(\vq_j\vp_j^{\,\intercal})^++(\vp_j\vq_j^{\,\intercal})^-,
\end{align}
where $\vp_j,\vq_j\in\R^{g+1}$ , with $(\vp_j)_0 >0$ and $\hat{\mathbf{C}}={\diag\{0,\bc_{k_\infty+1},\dots,\bc_{g+1},\bc_1,\dots,\bc_{k_\infty-1}\}}$ (with the obvious modification if $k_\infty=1$ or $k_\infty=g+1$) and $\vec\delta_{0}$ denotes the standard first basis vector of $\R^{g+1}$. $\tilde A_0$ and $\tilde B_0$ are projections of $A_0$ and $B_0$,
\begin{align*}
\tilde A_{0}= \Pi A_0 \quad \tilde B_{0}=\Pi B_0\Pi^*
\end{align*}
with $\Pi$ the block matrix $\Pi:=\left[0_{k_\infty\times (g+1-k_\infty)}\vert I_{k_\infty\times k_\infty }\right]$.
 We will refer to ${\{\vp_j,\vq_j\}}_{j\geq 1}$ as the GMP coefficients of $A$. While the precise structure will not be essential throughout the paper, we point out two things. First on the outermost diagonal of $A$ in each block there is only one non-vanishing entry, given by $(\vp_j)_0$, which is positive and which is at a different position depending on the position of $\infty$ in the sequence $\bC$. And secondly, in general as a self-adjoint matrix $B_j$ could depend on $(g+1)(g+2)/2$ parameters, but we see that in fact they only depend on $2(g+1)$. This is not that surprising due to their close relation to three-diagonal Jacobi matrices. A similar phenomena also appears for their unitary analogs \cite{ChrEichVan20}.
\begin{remark}\label{rem:GMPstructure}
	For later reference, we provide an alternative point of view on the block structure of $A$. The structure provided above is chosen so that $c_{k_\infty}$ is at the first diagonal position of the $B$-blocks. Recall also that to these blocks we attached a column $\vec p$ (with positive first entry $(\vp)_0 >0$)  to the right and a row $\vp_j^{\,\intercal}$ at the bottom. If, instead of viewing this as a block matrix structure with blocks of size $(g+1)\times (g+1)$, we view this structure as overlapping blocks of size $(g+2)\times (g+2)$ which overlap at the positions of $c_{k_\infty}$, then those would contain all non vanishing entries of the GMP matrix (i.e., it would also include the vector $\vp_j$). Moreover, the positive entries are exactly at the upper right and the lower left corner of the bigger block. Now placing the window of size  $(g+1)\times (g+1)$ on the top of the bigger block corresponds to the structure presented above. We will encounter in Section \ref{sec:CNCondition} that in other settings it may be more natural to place the block at the lower corner, and in this case the $B$ blocks will have structure similar to \eqref{eq:structureGMP}.
\end{remark}

Now the various notations for the off-diagonal blocks $A_j$, the vectors $\vp_j$ which determine them, and the coefficients $\beta_j$  defined in \eqref{betaDefn} are related as
\[
\beta_j =  \langle e_{j(g+1)+k_\infty} , A e_{(j+1)(g+1)+k_\infty} \rangle = (A_{j})_{00} = (\vp_{j})_0.
\]
The coefficients $\beta_j$ are a special case of the coefficients  $\Lambda_n$ defined for $n=j(g+1)+k$, $1\leq k\leq g+1$ as
\begin{equation} \label{LambdaDefn}
\Lambda_n = \begin{cases}
 \langle e_{j(g+1)+k}, (\bc_k-A)^{-1}e_{(j+1)(g+1)+k} \rangle & k\ne k_\infty \\
 \langle e_{j(g+1)+k}, Ae_{(j+1)(g+1)+k} \rangle & k=k_\infty
 \end{cases}
\end{equation}
Namely $\beta_j = \Lambda_{j(g+1) + k_\infty}$, and the coefficients $\Lambda_{j(g+1)+\ell}$ for $k \neq k_\infty$ instead occur as outermost diagonal coefficients for the GMP matrix $(\bc_k - A)^{-1}$. In our later applications to the discriminant of $A$, both the coefficients of $A$ and of its resolvents will appear, so we will work with $\Lambda_n$ throughout.

Next, we connect the coefficients \eqref{LambdaDefn}  to the solutions of the $L^2$-extremal problem \eqref{eq:L2extremal}.

\begin{lemma}\label{claim:1}
For all $n$,
\begin{align}\label{eq:kappaMatrixEntry}
\frac{\kappa_{n}}{\kappa_{n+g+1}}=\Lambda_n.
\end{align}
\end{lemma}

\begin{proof}
Let $n = j(g+1)+k$. By self-adjointness,
\begin{align*}
\Lambda_n  = \langle e_{n},r_k(A) e_{n+g+1}\rangle=
\langle r_k \tau_{n} , \tau_{n+g+1}\rangle =\langle \kappa_{n}r_{n+g+1} +h  ,\tau_{n+g+1} \rangle
\end{align*}
for some $h\in \cL_{n+g}$.  By orthogonality,
$\langle \tau_{n+g+1}, h\rangle = 0$, so $\langle \tau_{n+g+1},r_{n+g+1} \rangle=\frac1{\kappa_{n+g+1}}$ implies that
\[
\Lambda_n = \langle \tau_{n+g+1}, \kappa_{n} r_{n+g+1} +h \rangle=\frac{\kappa_{n}}{\kappa_{n+g+1}}. \qedhere
\]
\end{proof}

We now adapt to GMP matrices ideas from the theory of  regularity for Jacobi matrices \cite{Simon07}.

\begin{lemma}\label{lem:lowerbound}
Let $A\in \A(\bC)$. For all $j\geq 1$, $\|\vec p_j\|\leq \lVert A \rVert$.
\end{lemma}

\begin{proof}
Fix $j \ge 1$ and denote $n= j (g+1) + k_\infty$. For any $\ell = 0,\dots,g$,
\[
(p_j)_\ell=\langle e_{n - g - 1+\ell},Ae_{n} \rangle = \int \tau_{n-g-1+\ell}(x) x\tau_n(x) d\mu(x).
\]
Since the vectors $\tau_{n-g-1+\ell}$ are orthonormal, by the Bessel inequality,
\[
\lVert \vp_j \rVert^2 \le \int \lvert x \tau_n(x) \rvert^2 d\mu(x) \le \lVert A \rVert^2 \int \lvert \tau_n(x) \rvert^2 d\mu(x) = \lVert A \rVert^2
\]
since $\lVert A \rVert = \sup_{x\in\supp \mu} \lvert x \rvert$.
\end{proof}

\begin{lemma}\label{lem:positivity}
For $z \in \bbC \setminus \bbR$,
\begin{equation}\label{25jul1}
\liminf_{n\to\infty}\frac{1}{n}\log|\tau_{n}(z)| > 0.
\end{equation}
\end{lemma}

\begin{proof}
We adapt the proof of  \cite[Proposition 2.2]{Simon07}. It suffices to prove \eqref{25jul1} along the subsequences $n(j) = j(g+1) + k$, $j\to\infty$, for $1\le k \le g+1$. Moreover, due to $\overline{\bbR}$-preserving conformal invariance, it suffices to fix $k$ and prove
\begin{equation}\label{25jul1X}
\liminf_{j\to\infty}\frac{1}{n(j)}\log|\tau_{n(j)}(z)| > 0
\end{equation}
under the assumption that $\bc_k = \infty$. This allows us to use the associated GMP matrix $A \in \A(\bC)$.

Note that for any $m$, since $\{\tau_\ell\}_{\ell=0}^\infty$ is an orthonormal basis of $L^2(d\mu)$,
\[
\sum_\ell A_{m\ell} \tau_\ell(z)= \sum_\ell \langle z \tau_m(z), \tau_\ell(z) \rangle \tau_\ell(z)  = z \tau_m(z).
\]
This equality holds in $L^2(d\mu)$, but since all functions are rational, it also holds pointwise. Thus, if we fix $z \in \bbC \setminus \bbR$, the sequence $\vec \varphi = \{ \tau_\ell(z) \}_{\ell=0}^\infty$ is a formal eigensolution for $A$ at energy $z$, i.e. $(A-z)\vec{\varphi}=0$ componentwise. Since $A$ is represented as a block tridiagonal matrix, let us also write $\vec{\varphi}$ in a matching block form, as $\vec\varphi^\top = \begin{bmatrix} \vec u_0^\top & \vec u_1^\top & \vec u_2^\top & \dots \end{bmatrix}$ where
\begin{align*}
&\vec u_{0}^\top =\begin{bmatrix}
\tau_{0}(z)&\dots & \tau_{k-1}(z)\end{bmatrix},\quad \vec u_{j}^\top =\begin{bmatrix}
\tau_{n(j-1)-1}(z)&\dots &  \tau_{n(j)-1}(z)\end{bmatrix},\quad j\geq 1.
\end{align*}
We also consider the projection of $\vec{\varphi}$ onto the first $j+1$ blocks,
\[
\vec\varphi_{j}^\top = \begin{bmatrix}
\vec u_0^\top &\dots & \vec u_j^\top  &0 & \dots\end{bmatrix},
\]
and compute $(A-z) \vec\varphi_j$. By the block tridiagonal structure of $A$, for $m< n(j-1)$ we have  $\langle e_m,(A-z)\vec{\varphi}_{j}\rangle=0$. For $0\leq \ell\leq g$, we have
\[
\langle e_{n(j-1)+\ell},(A-z)\vec{\varphi}\rangle-\langle e_{n(j-1)+\ell},(A-z)\vec{\varphi}_{j}\rangle=(p_j)_\ell\tau_{n(j)}(z)
\]
so that $\langle e_{n(j-1)+\ell},(A-z)\vec{\varphi}_{j}\rangle=-(p_j)_\ell\tau_{n(j)}(z)$. Moreover,
\[
\langle e_{n(j)},(A-z)\vec\varphi_{j}\rangle=\langle e_{n(j)},A\vec\varphi_{j}\rangle=(\vp_j)^* u_j(z).
\]
For $m>n(j)$, we again have $\langle e_{m},(A-z)\vec\varphi_{j}\rangle=0 $.  In conclusion,  $(A-z) \vec\varphi_j$ has only two nontrivial blocks,
\begin{align*}
((A-z)\vec\varphi_j )^\top =\begin{bmatrix}
0 & \dots & 0 &-(\vp_j\tau_{n(j)}(z))^\top &  ((\vp_j)^* u_j)^\top & 0 & \dots
\end{bmatrix}.
\end{align*}
In particular, we can compute
\begin{equation}\label{7aug1}
\langle\vec\varphi_j ,(A-z)\vec\varphi_j \rangle = -\vec{u}_j^*\tau_{n(j)}(z) \vp_j.
\end{equation}
Since $A$ is self-adjoint and $\vec\varphi_j \in \ell^2(\bbN_0)$, by a standard consequence of the spectral theorem \cite[Lemma 2.7.]{TeschlMathMedQuant},\begin{align*}
\lvert \Im z \rvert \|\vec\varphi_{j}\|^2 \le |\langle\vec\varphi_{j},(A-z)\vec\varphi_{j}\rangle|.
\end{align*}
Using \eqref{7aug1} and the Cauchy--Schwarz inequality gives
\begin{align*}
\lvert \Im z \rvert \sum_{m=0}^{j}\|\vec{u}_m \|^2
&\leq |\tau_{n(j)}(z)|\|\vp_j\| \|\vec{u}_j\|.
\end{align*}
By Lemma~\ref{lem:lowerbound}, with $C =\lvert \Im z \rvert / \lVert A \rVert$,
\begin{align}\label{eq:1Aug3}
C\sum_{m=0}^{j}\|\vec{u}_m\|^2\leq |\tau_{n(j)}(z) |\|\vec{u}_j\|.
\end{align}
Applying the AM-GM inequality to the right-hand side of \eqref{eq:1Aug3} gives
\[
|\tau_{n(j)}(z)|\|\vec{u}_j(z)\|\leq \frac12\left( C \|\vec{u}_j(z)\|^2+C^{-1} |\tau_{n(j)}(z)|^2\right)
\]
which together with \eqref{eq:1Aug3} implies
\begin{align}\label{eq:May5}
|\tau_{n(j)}(z)|^2\geq C^2 \sum_{m=0}^{j}\|\vec{u}_m \|^2.
\end{align}
Since $|\tau_{n(j)}(z)|^2\leq \|\vec{u}_{j+1} \|^2$, this implies that
\[
\sum_{m=0}^{j+1}\|\vec{u}_m \|^2\geq \left(1+C^2\right)\sum_{m=0}^{j}\|\vec{u}_m \|^2.
\]
Since $\lVert \vec{u}_0 \rVert \ge \lvert \tau_0(z) \rvert = 1$, this implies by induction that
\begin{align*}
\sum_{m=0}^{j}\|\vec{u}_m \|^2\geq\left(1+ C^2\right)^j.
\end{align*}
Combining this with \eqref{eq:May5} gives a lower bound on $\lvert \tau_{n(j)}(z) \rvert$ which implies \eqref{25jul1X}.
\end{proof}

The estimates in the previous proof also lead to the following:
\begin{corollary}\label{cor:equivalence}
For any $z\in \C\setminus \R$, the quantities
\[
\liminf_{j\to \infty}\frac 1{j(g+1)+k} \log | \tau_{j(g+1)+k}(z) |, \qquad \limsup_{j\to \infty}\frac 1{j(g+1)+k} \log | \tau_{j(g+1)+k}(z) |
\]
are independent of $k \in \{1,\dots, g+1\}$.
\end{corollary}

\begin{proof}
Assume $j\ge 1$. For $k-g-1 \le \ell \le k-1$, the estimate \eqref{eq:May5} gives
\[
|\tau_{j(g+1)+k}(z)|^2\geq C^2\|\vec{u}_j \|^2\geq C^2|\tau_{j(g+1)+\ell}(z)|^2
\]
which implies
\begin{align}\label{eq:ineq1}
\liminf_{j\to \infty}\frac 1{j(g+1)+k} \log | \tau_{j(g+1)+k}(z) |\geq
\liminf_{j\to \infty}\frac 1{j(g+1)+\ell} \log | \tau_{j(g+1)+\ell}(z) |
\end{align}
and
\begin{align}\label{eq:ineq2}
\limsup_{j\to \infty}\frac 1{j(g+1)+k} \log | \tau_{j(g+1)+k}(z) |\geq
\limsup_{j\to \infty}\frac 1{j(g+1)+\ell} \log | \tau_{j(g+1)+\ell}(z) |.
\end{align}
Clearly, the right-hand sides don't change if $\ell$ is shifted by $g+1$, so \eqref{eq:ineq1}, \eqref{eq:ineq2} hold for all $k, \ell \in \{1,\dots,g+1\}$ with $k \neq \ell$. By symmetry, since the roles of $k, \ell$ can be switched, we conclude that equality holds in  \eqref{eq:ineq1}, \eqref{eq:ineq2}.
\end{proof}

\section{Growth rates of orthonormal rational functions}

In this section, we will combine the positivity \eqref{25jul1} with potential theory techniques  in order to study exponential growth rates of orthonormal rational functions. Our main conclusions will be conformally invariant, but our proofs will use potential theory arguments and objects such as the logarithmic potential of a finite measure $\nu$,
\[
\Phi_\nu(z) = \int \log \lvert z -x \rvert d\nu(x),
\]
which is well defined  when $\supp \nu$ does not contain $\infty$.

\begin{theorem}\label{thm:June9}
Fix $1\leq k\leq g+1$ and denote by $I$ the connected component of $\overline{\R}\setminus \supp \mu$ containing $\bc_k$.
	Suppose there is a subsequence $n_\ell = j_\ell(g+1) + k$ such that $\wlim_{\ell\to \infty}\nu_{n_\ell}= \nu$ and $\frac{1}{n_\ell}\log \kappa_{n_\ell} \to \alpha \in \R \cup \{-\infty,+\infty\}$ as $\ell \to \infty$.
  Then uniformly on compact subsets of $(\overline{\C} \setminus \overline{\R}) \cup (I \setminus \{\bc_k\})$, we have
	\begin{align*}
		 h(z) :=\lim\limits_{\ell\to\infty}\frac{1}{n_\ell}\log|\tau_{n_\ell}(z)|.
	\end{align*}
	The function $h$ is determined by $\nu$ and $\alpha$; in particular, if $\bc_k = \infty$,
	\begin{equation}\label{24jul1}
h(z) =\alpha + \Phi_{\nu}(z)- \frac 1{g+1} \sum_{\substack{m=1\\ m\ne k}}^{g+1}\log|\bc_m-z|.
\end{equation}
Moreover,
\begin{enumerate}[(a)]
\item\label{it:impossible} $\alpha = -\infty$ is impossible;
\item\label{it:infty} If $\alpha = +\infty$, the limit is $h =+  \infty$;
\item\label{it:limits} If $\alpha \in \bbR$, the limit $h$ extends to a positive harmonic function on $\overline{\bbC}\setminus ( \E\cup \{ \bc_1,\dots,\bc_{g+1} \})$ such that
\begin{align}
h(z) & = - \frac 1{g+1} \log \lvert \bc_m - z \rvert + O(1), \quad z \to \bc_m \neq \infty \label{10aug3} \\
h(z) & = \frac 1{g+1} \log \lvert z \rvert + O(1), \quad z \to \bc_m = \infty. \label{10aug4}
\end{align}
\end{enumerate}
\end{theorem}

\begin{proof}
By using $\overline{\bbR}$-preserving conformal invariance, we can assume without loss of generality that $\bc_k = \infty$. We will use the representation \eqref{LnRndefn} of the subspace $\cL_n$. For $n= j(g+1)+k$, counting degrees of the poles leads to
\[
\tau_n = \frac {P_n}{R_n}, \qquad R_n(z) = \prod_{m=1}^{k-1}(\bc_m-z)\prod_{\substack{m=1\\m\ne k}}^{g+1}(\bc_m -z)^{j},
\]
with $\deg P_n = n$. This may not be the minimal representation of $\tau_n$, but by the proof of Lemma~\ref{lem:zeroesofnumeratornew}, the only possible cancellations are simple factors $(\bc_m - z)$ for each $m \neq k$, so we get the minimal representation $\tau_n(z)=P(z)/Q(z)$ with
\begin{align*}
P(z)=\kappa_n\prod_{w:\tau_n(w)=0}(z-w),\quad Q(z)=
\prod_{\substack{m=1\\ m\ne k}}^{g+1}(\bc_m -z)^{j + \delta_{m,j}}
\end{align*}
where $\lvert \delta_{m,j} \rvert \le 1$ for each $j$. All that matters is that $\delta_{m,j} / j \to 0$ as $j\to\infty$. It will be useful to turn this rational function representation into a kind of Riesz representation,
\begin{equation}\label{10aug1}
\log \lvert \tau_n(z) \rvert = \log \kappa_n + n \int \log \lvert x - z \rvert d\nu_n(x) - \sum_{\substack{1\le m \le g+1 \\ m \neq k}} (j+\delta_{m,j}) \log \lvert \bc_m - z \rvert.
\end{equation}

Since $\bc_k = \infty$, note that $K = \overline{\bbR} \setminus I$ is a compact subset of $\bbR$.  Denote $\Omega =  \bbC \setminus K$.  For any $z \in \Omega$, the map $x\mapsto \log|x-z|$ is continuous on $K$, so $\Phi_{\nu_{n_\ell}}(z) \to \Phi_{\nu}(z)$ as $\ell \to \infty$. In fact, convergence is uniform on compact subsets of $\Omega$: since $\supp(\nu_{n_{\ell}})\subset K$ and $\nu_{n_\ell}(K) \le 1$ for all $\ell$,  the estimate
\[
\log \left\lvert \frac{x-z_1}{x-z_2} \right\rvert \le \log \left( 1 + \frac{\lvert z_1 - z_2 \rvert}{\dist(z_2,K)} \right) \le \frac{\lvert z_1 - z_2 \rvert}{\dist(z_2, K)}, \qquad z_1,z_2 \in \Omega
\]
implies uniform equicontinuity of the potentials $\Phi_{n_\ell}$ on compact subsets of $\Omega$, and the Arzel\`a--Ascoli theorem implies uniform convergence on compacts.

Note that \eqref{it:infty} follows from \eqref{24jul1}. By Corollary~\ref{cor:precompMeasure}, $\supp \nu \subset \E$ and
$\Phi_{\nu}(z)$ is harmonic on $\C\setminus \E$, so the right hand side extends to a harmonic function on $\C\setminus( \E\cup \{\bc_1,\dots,\bc_{g+1}\})$ and we denote this extension also by $h$. By Lemma ~\ref{lem:positivity}, $h$ is positive on $\C_+\cup \C_-$, so $\alpha \neq -\infty$; moreover, by the mean value property, $h$ is positive on $\C\setminus ( \E\cup \{\bc_1,\dots,\bc_{g+1}\})$.

The remaining asymptotic properties follow from \eqref{24jul1}. Under the assumption $\bc_k=\infty$, $\supp \nu$ is a compact subset of $\R$, and
$\Phi_\nu(z)=\log|z|+O(1)$, $z\to \infty$. It then follows that $h(z)=\frac{1}{g+1}\log|z|+O(1)$ as $z\to\infty$. Of course, $h(z)=-\frac{1}{g+1}\log|z-\bc_m |+O(1)$ near each $\bc_m \neq \bc_k$.
\end{proof}
The previous theorem motivates interest in positive harmonic functions on $\overline{\bbC} \setminus (\E \cup \{ \bc_1,\dots,\bc_{g+1}\})$. If $\E$ is polar, by Myrberg's theorem \cite[Theorem 5.3.8]{Classpotential}, any such function is constant. If $\E$ is not polar, knowing the asymptotic behavior of $h$ at the poles, positivity of $h$ improves to the following lower bound on $h$. The following Lemma reflects a standard minimality property of the Green function \cite[Section VII.10]{Doob}.

\begin{lemma}\label{lem:Greenminimality}
Assume that $\E$ is a nonpolar closed subset of $\overline{\R}$. Let $h$ be a positive superharmonic function on $\overline{\C}\setminus ( \E\cup \{ \bc_1,\dots, \bc_{g+1}  \} )$. Suppose $h(z)+\frac{1}{g+1}\log|z-\bc_k|$ has an existent limit at $\bc_k$ for each finite $\bc_k$, and
$h(z)-\frac1{g+1}\log|z|$ has an existent limit at $\infty$ if one of the $\bc_k=\infty$. Then
\begin{align}\label{eq:inequalityFunctions}
h(z)\geq \cG_\E(z,\bC)
\end{align}
for $z\in \overline{\bbC}\setminus \E$. For $1\leq k\leq g+1$, define
\begin{align*}
\a_k=\begin{cases}
\lim\limits_{z\to\bc_k}(h(z)+\frac1{g+1}\log|z-\bc_k|),&\bc_k\ne \infty\\
\lim\limits_{z\to\infty}(h(z)-\frac{1}{g+1}\log|z|),& \bc_k=\infty
\end{cases}
\end{align*}
Then
\begin{align}\label{eq:inequalityConstants}
\a_k\geq \frac{\log\lambda_k}{g+1}
\end{align}
\end{lemma}
\begin{proof}
We will use a stronger, q.e. version of the maximum principle \cite[Thm 3.6.9]{RansfordPotTheorie}. Define
\[
\tilde h(z):=\cG_\E(z,\bC)-h(z),
\]
which is bounded at $\bc_k$ for $1\leq k\leq g+1$ and so extends to a subharmonic function on $\overline{\C}\setminus \E$. Since $\cG_\E$ vanishes q.e. on $\E$,
we have for q.e. $t\in \E$,
\[
\limsup_{z\to t} \tilde{h}(z)=-\liminf_{z\to t} h(z)\leq  0.
\]
Now we show $\tilde h$ is bounded above on $\overline\C\setminus \E$. Let $\mathcal U$ be a union of small neighborhoods containing the points $\bc_k$ in $\overline{\C}\setminus \E$. By the definition of the Green function, $\cG_\E(z,\bC)$ defines a harmonic and bounded function on $\overline{\bbC}\setminus(\E\cup \mathcal U)$. That is, there exists $M$ such that for all $z\in\overline{\bbC}\setminus (\mathcal U\cup \E)$ we have
\begin{align*}
\cG_\E(z,\bC)\leq M.
\end{align*}
Since $h\geq 0$, it follows on $\overline{\bbC}\setminus (\mathcal U\cup \E)$ that
\begin{align*}
\tilde h(z)= \cG_\E(z,\bC)-h(z)\leq \cG_\E(z,\bC)\leq M.
\end{align*}
On the other hand, by properties of the Green functions we have
\begin{align*}
\frac{\log\lambda_k}{g+1}=\begin{cases}
\lim\limits_{z\to\bc_k}(\cG_\E(z,\bC)+\frac1{g+1}\log|z-\bc_k|),&\bc_k\ne \infty\\
\lim\limits_{z\to\infty}(\cG_\E(z,\bC)-\frac{1}{g+1}\log|z|),& \bc_k=\infty
\end{cases}
\end{align*}
Then, by assumption, for $1\leq k\leq g+1$, $\tilde h(z)=\frac{\log\lambda_k}{g+1}-\alpha_k+o(1)$ as $z\to \bc_k$ and, in particular, the difference is bounded in a small neighborhood of $\bc_k$. Thus, $\tilde h$ is bounded above on $\overline{\C}\setminus \E$.

So, by the maximum principle $\tilde h\leq 0\implies \cG_\E(z,\bC)\leq h(z)$ on $\overline\C\setminus \E$.
Since $0\geq \lim_{z\to\bc_k}\tilde h(z)=\frac{\log\lambda_k}{g+1}-\alpha_k$, we have \eqref{eq:inequalityConstants}.
\end{proof}

\begin{lemma}\label{lem:greenequality}
Under the same assumptions as Lemma~\ref{lem:Greenminimality}, the following are equivalent:
\begin{enumerate}[(i)]
\item\label{it:equalityall} Equality in \eqref{eq:inequalityConstants} for all $k$ with $1\leq k\leq g+1$
\item\label{it:equality1} Equality in \eqref{eq:inequalityConstants} for a single $k$ with $1\leq k\leq g+1$
\item\label{it:equalityfunction} Equality holds in \eqref{eq:inequalityFunctions}
\end{enumerate}
\end{lemma}
\begin{proof}
\eqref{it:equalityall}$\implies$ \eqref{it:equality1} is trivial. Suppose then \eqref{it:equality1}; with the notation of the previous lemma, by assumption, $\tilde h(\bc_k)=0$ and $\tilde h$ achieves a global maximum. By the maximum principle for subharmonic functions \cite[Theorem 2.3.1]{RansfordPotTheorie}, $\tilde h\equiv 0$ on $\overline{\C}\setminus \E$.
Finally, if \eqref{it:equalityfunction} holds, then evaluating $\tilde h(\bc_k)$ for each $1\leq k\leq g+1$ yields \eqref{it:equalityall}.
\end{proof}

We will now prove Theorems \ref{thm11} and \ref{thm12}.

\begin{proof}[Proof of Theorem~\ref{thm11}]
Using conformal invariance, we take $\bc_k=\infty$.
Fix $z\in \overline\C\setminus \overline\R$ and select a sequence $(n_\ell)_{\ell=1}^\infty$ such that
\[
\liminf_{n\to \infty}\frac{1}{n}\log|\tau_n(z)|=\lim_{\ell\to \infty}\frac{1}{n_{\ell}}\log|\tau_{n_\ell}(z)|.
\]
By precompactness of the $(\nu_{n})$, we may pass to a further subsequence, which we denote again by $(n_\ell)_{\ell=1}^\infty$, so that
$\wlim_{\ell\to\infty}\nu_{n_\ell}= \nu$ and $\frac 1{n_\ell} \log \kappa_{n_\ell} \to \alpha$ for some $\nu$ and $\alpha$. Then for $h$ as in Theorem~\ref{thm:June9},
\[
\lim_{\ell\to \infty}\frac{1}{n_\ell}\log|\tau_{n_\ell}(z)|=h(z).
\]
on $\C\setminus \R$. If $\alpha=+\infty$, then there is nothing to show. Suppose $\alpha<\infty$. If $\E$ is not polar we apply \eqref{it:impossible} of Theorem~\ref{thm:June9} to find $\alpha\in \R$, and we may use \eqref{it:limits} of the same theorem
and Lemma~\ref{lem:Greenminimality} to conclude.

If instead $\E$ is polar, by Myrberg's theorem, $h$ is constant on $\C\setminus (\E\cup \{\bc_1,\dots,\bc_{g+1}\})$. Computing the limit at $\bc_k$ we see $h\equiv +\infty$. In particular, $\liminf_{n\to \infty}\frac{1}{n}\log|\tau_n(z)|=+\infty$ for $z\in \C\setminus \R$.
\end{proof}

\begin{proof}[Proof of Theorem~\ref{thm12}]

Fix $1\leq k\leq g+1$ and assume again by conformal invariance that $\bc_k = \infty$. Using precompactness of the measures $(\nu_{n})$, we find a subsequence $n_\ell=j_\ell(g+1)+k$ with
\[
\lim_{\ell\to \infty}\frac{1}{n_\ell}\log\kappa_{n_\ell}=\liminf_{j\to \infty}\frac{1}{n(j)}\log\kappa_{n(j)}=:\alpha
\]
and $\wlim_{\ell\to \infty}\nu_{n_\ell}=\nu$. If $\alpha=+\infty$, we are done.
Suppose then $\alpha<\infty$, then we have by Theorem~\ref{thm:June9} \eqref{it:impossible}, $\alpha\in \R$. Furthermore, if $\E$ is nonpolar, by \eqref{it:limits} and Lemma~\ref{lem:Greenminimality}, $h(z)\geq \cG_\E(z,\bC)$ on $\overline\C\setminus \E$.
In particular, by the representation \eqref{24jul1} we see that $\alpha=\lim_{z\to\infty}(h(z)-\frac{1}{g+1}\log|z|)$, and so \eqref{eq:inequalityConstants} yields the desired inequality.

If instead $\E$ is polar, by Theorem~\ref{thm11}, for each $z\in \C\setminus \R$,
\[
h(z)=\lim_{\ell\to \infty}\frac{1}{n_\ell}\log|\tau_{n_\ell}|\geq
\liminf_{n\to \infty}\frac{1}{n}\log|\tau_n(z)|=+\infty.
\]
and so by Theorem~\ref{thm:June9} \eqref{it:infty}, $\alpha=+\infty$.
\end{proof}

\section{Regularity}

We will begin by proving a version of Theorem~\ref{thm:CRegMain} for a fixed $k$.
\begin{lemma}\label{lemma:CRegMain}
Fix $k\in \{1,\dots,g+1\}$. Along the subsequence $n(j) = j(g+1) +k$, the following are equivalent:
	\begin{enumerate}[(i)]
		\item\label{it:CRegConstantsk} $\lim_{j\to\infty} \kappa_{n(j)}^{1/n(j)} =  \lambda_k^{1 / (g+1)}$;
\item\label{it:CRegBoundaryk} For q.e. $z\in\E$, we have
$\limsup_{j\to\infty} |\tau_{n(j)}(z)|^{1/n(j)} \leq 1$;
\item \label{it:upperhalfplanek} For some $z\in\C_+$, $\limsup_{j\to\infty} |\tau_{n(j)}(z)|^{1/n(j)} \leq e^{\cG_\E(z,\bC)}$;
\item\label{it:planek}	For all $z\in\C$,
$\limsup_{j\to\infty}  |\tau_{n(j)}(z)|^{1/n(j)} \leq e^{\cG_\E(z,\bC)}$;
\item  \label{it:uniformlimitk} Uniformly on compact subsets of $\C\setminus\R$,
$\lim_{j\to\infty} |\tau_{n(j)}(z)|^{1/n(j)}= e^{\cG_\E(z,\bC)}$.
\end{enumerate}
\end{lemma}

\begin{proof}
    Using conformal invariance, we will assume throughout the proof that $\bc_k=\infty$.
    First, suppose that $\E$ is polar. In this case \eqref{it:CRegBoundaryk} is vacuous, and since $\cG_\E\equiv +\infty$,  \eqref{it:upperhalfplanek} and \eqref{it:planek} are trivially true. Since $\lambda_k = +\infty$, \eqref{it:CRegConstantsk} follows from Theorem~\ref{thm12}.     As in the proof of Theorem~\ref{thm:June9}, weak convergence of measures implies uniform on compacts convergence of their potentials. Thus, since $\nu_n$ are a precompact family, so are $\Phi_{\nu_n}$. Thus,     the convergence $\lim_{j\to\infty} \frac 1{n(j)} \log \kappa_{n(j)} = +\infty$ implies that $\lim_{j\to\infty} \frac 1{n(j)} \log \lvert \tau_{n(j)}(z) \rvert = +\infty$ uniformly on compact subsets of $\bbC \setminus \bbR$,
    so \eqref{it:uniformlimitk} holds.

    For the remainder of the proof, we will assume $\E$ is not polar.    Moreover, we will repeatedly use the fact that if any subsequence of a sequence in a topological space has a further subsequence which converges to a limit, then the sequence itself converges to this limit. In particular, when concluding \eqref{it:uniformlimitk}, we apply this fact in the Fr\'echet space of harmonic functions on $\C\setminus \R$ with the topology of uniform convergence on compact sets.

	  \eqref{it:upperhalfplanek}$\implies$\eqref{it:uniformlimitk}: Given a subsequence of $n(j)=j(g+1)+k$, using precompactness of the measures $\nu_n$, we pass to a further subsequence $n_\ell=j_\ell(g+1)+k$ with $\wlim_{\ell\to\infty}\nu_{n_\ell}=\nu$
    and $\lim_{\ell\to \infty}\frac{1}{n_\ell}\log\kappa_{n_\ell}=:\alpha$, with $\alpha$ real or infinite. By Theorem~\ref{thm:June9}, uniformly on compact subsets of $\C\setminus \R$,
    \[
    h(z)=\lim_{\ell\to \infty} \frac 1{n_\ell} \log|\tau_{n_{\ell}}(z)|
    \]
    with $h$ given by \eqref{24jul1}. Using the assumption, for some $z_0\in \C_+$, we have
    \begin{align*}
      h(z_0)\leq \limsup_{j\to \infty}\frac{1}{n(j)}\log|\tau_{n(j)}(z_0)|<\infty.
    \end{align*}
    So, by Theorem~\ref{thm:June9}, $\alpha\in\R$ and $h$ has a harmonic extension to $\overline\C\setminus (\E\cup \{\bc_1,\dots,\bc_{g+1}\})$. Furthermore, by Lemma~\ref{lem:Greenminimality}, $h\geq \cG_\E$. By assumption, we have the opposite inequality at $z_0\in\C_+$, and so, by the maximum principle for harmonic functions,
    $h=\cG_\E$ on $\C\setminus (\E\cup\{ \bc_1,\dots,\bc_{g+1}\})$, and in particular on $\C\setminus \R$. Thus, we have \eqref{it:uniformlimitk}.

	  \eqref{it:uniformlimitk}$\implies$\eqref{it:planek}:
    For $z\in \{ \bc_1,\dots,\bc_{g+1}\}$, $\cG_\E(z,\bC)=+\infty$ and there is nothing to show. Fix $z\in \C\setminus \{\bc_{1},\dots,\bc_{g+1}\}$ and let $n_\ell=j_\ell(g+1)+k$ be a subsequence with $\lim_{\ell\to \infty}\frac{1}{n_\ell}\log|\tau_{n_\ell}(z)|=\limsup_{j\to \infty}\frac1{n(j)}\log|\tau_{n(j)}(z)|$.
    By passing to a further subsequence, we may assume $\wlim_{\ell\to \infty}\nu_{n_\ell}=\nu$,
    and $\lim_{\ell\to \infty}\frac{1}{n_\ell}\log\kappa_{n_\ell}=:\alpha$ where $\alpha$ is real or infinite.     By the assumption, we have $h=\lim_{\ell\to \infty}\frac{1}{n_\ell}\log|\tau_{n_\ell}|=\cG_\E$ on $\C\setminus \R$.
    So, by \eqref{it:impossible} and \eqref{it:infty},
    $\alpha\in\R$ and $h$ extends to a harmonic function on $\overline{\C}\setminus(\E\cup\{ \bc_1,\dots,\bc_{g+1}\})$. By the representation \eqref{24jul1}, we may extend $h$ subharmonically to $\overline{\C}\setminus \{\bc_1,\dots,\bc_{g+1}\}$.
    On this set, $\cG_\E$ is also subharmonic, so, by the weak identity principle \cite[Theorem 2.7.5]{RansfordPotTheorie},
    $h=\cG_\E$ on $ \overline\C\setminus \{ \bc_1,\dots,\bc_{g+1}\}$. Thus, by the principle of descent \cite[A.III]{StahlTotik92}, we have
    \begin{align}\label{eq:descent}
    \lim_{\ell\to \infty}\frac1{n_\ell}\log|\tau_{n_\ell}(z)|\leq h(z)=\cG_\E(z,\bC)
    \end{align}
    and \eqref{it:planek} follows.

		\eqref{it:uniformlimitk}$\implies$ \eqref{it:CRegConstantsk}:      Given a subsequence of $n(j)=j(g+1)+k$, we use precompactness of the $\nu_n$ to pass to a further subsequence $n_\ell=j_\ell(g+1)+k$ with $\lim_{\ell\to\infty }\frac{1}{n_\ell}\log\kappa_{n_\ell}=:\alpha\in \R\cup\{ -\infty,+\infty\}$ and $\wlim_{\ell\to \infty}\nu_{n_\ell}=\nu$.
    Then in the notation of Theorem~\ref{thm:June9}
    and by assumption, for a $z\in \C\setminus \R$
    \begin{align*}
      \lim_{\ell\to \infty}\log|\tau_{n_\ell}(z)|=h(z)=\cG_\E(z,\bC).
    \end{align*}
    So by, Lemma~\ref{lem:greenequality}, $\alpha=\frac{\log\lambda_k}{g+1}$. Thus, $\lambda_k^{1/(g+1)}$ is the only accummulation point of $\kappa_{n(j)}^{1/n(j)}$ in $\bbR \cup \{-\infty,+\infty\}$ and we have \eqref{it:CRegConstantsk}.

    \eqref{it:CRegConstantsk}$\implies$\eqref{it:uniformlimitk}:     As before, we fix a subsequence of $n(j)=j(g+1)+k$ and use precompactness to pass to a further subsequence $n_\ell=j_\ell(g+1)+k$ with $\wlim_{\ell\to\infty}\nu_{n_\ell}=\nu$.
    Then, by Theorem~\ref{thm:June9} and in the notation introduced there, uniformly on compact subsets of $\C\setminus\R$,
    \[
    \lim_{\ell\to\infty}\frac{1}{n_\ell}\log|\tau_{n_\ell}(z)|=h(z)
    \]
    where $h$ is given by \eqref{24jul1} with $\alpha=\frac{\log\lambda_k}{g+1}$.
    Thus, by Lemma~\ref{lem:greenequality} \eqref{it:equality1}, $h(z)=\cG_\E(z,\bC)$ on $\C\setminus \R$. Since the initial subsequence was arbitrary, we have \eqref{it:uniformlimitk}.

		\eqref{it:planek}$\implies$ \eqref{it:CRegBoundaryk}: Recalling that the Green function vanishes q.e. on $\E$, the claim follows.

		\eqref{it:CRegBoundaryk}$\implies$ \eqref{it:uniformlimitk}:
    Fixing a subsequence of $n(j)$, we again use precompactness to select a further subsequence $n_\ell=j_\ell(g+1)+k$ such that $\wlim_{\ell\to\infty}\nu_{n_\ell}=\nu$ and $\lim_{\ell\to\infty}\frac{1}{n_\ell}\log\kappa_{n_\ell}=:\alpha$, $\alpha\in \R\cup \{ -\infty,+\infty\}$.
    By the upper envelope theorem, there is a polar set $X_1\subset \C$ such that on $\C\setminus X_1$, $\limsup_{\ell\to\infty}\Phi_{\nu_{n_\ell}}=\Phi_{\nu}$. Now, we let $X_2:=\{ t\in \E:\limsup_{n\to \infty}\frac1n\log|\tau_n(t)|>0\}$, which is polar by assumption, and
    $X_3:=\{z\in\C: \Phi_\infty(z)=-\infty\}$, which is polar by \cite[Theorem 3.5.1]{RansfordPotTheorie}. Then, for a $t\in \E\setminus (X_1\cup X_2\cup X_3)$, we have
    \begin{align*}
      \alpha\leq \limsup_{n\to \infty}\frac{1}{n}\log|\tau_{n}(t)|-\Phi_\nu(t)+\frac1{g+1}\sum_{\substack{m=1\\m\ne k}}^{g+1}\log|\bc_m-t|<\infty.
    \end{align*}
    So $\alpha\in \R$ by Theorem~\ref{thm:June9} \eqref{it:impossible}. Thus, by \eqref{it:limits} of the same theorem, uniformly on compact subsets of $\C\setminus \R$
    \[
    h(z)=\lim_{\ell\to \infty}\frac{1}{n_\ell}\log|\tau_{n_\ell}(z)|
    \]
    and $h$ extends to a positive harmonic function on $\overline{\C}\setminus (\E\cup \{ \bc_1,\dots,\bc_{g+1}\})$ with logarithmic poles at each of the $\bc_m$. So,
    $h- \cG_\E$ extends to a harmonic function on $\overline{\C}\setminus \E$, and $h- \cG_\E\geq 0$ there by Lemma~\ref{lem:Greenminimality}.
    We now show that in fact $h=\cG_\E$ using the stronger, q.e. maximum principle.

    We use the equality in \eqref{24jul1} to extend $h$ to a subharmonic function on $\C\setminus \{ \bc_1,\dots,\bc_{g+1}\}$.
    By the upper envelope theorem and the assumption again, for $t\in \E\setminus (X_1\cup X_2)$
    \[
    h(t)=\limsup_{\ell\to \infty}\frac{1}{n_\ell}\log|\tau_{n_\ell}(t)|\leq 0.
    \]
    Then,
    for these $t$, since $\cG_\E$ is positive, we have
    \[
    \limsup_{\substack{z\to t\\ z\in \C\setminus \E}}\left( h(z)-\cG_\E(z,\bC)\right)\leq \limsup_{\substack{z\to t\\ z\in \C\setminus \E}}h(z)\leq h(t)\leq 0
    \]
    by upper semicontinuity. So, $\limsup_{\substack{z\to t\\ z\in \C\setminus \E}}\left( h(z)-\cG_\E(z,\bC)\right)\leq 0$ for q.e. $t\in \E$.

    Since $h$ is upper semicontinuous on the compact set $\E$, there is an $M$ so that $\sup_{t\in \E}h(t)\leq M$. As in the above, now for any $t\in\E$, we have
    \[
    \limsup_{\substack{z\to t\\ z\in \C\setminus \E}}\left( h(z)-\cG_\E(z,\bC)\right)\leq \limsup_{\substack{z\to t\\ z\in \C\setminus \E}}h(z)\leq h(t)\leq M.
    \]
    So, there is a neighborhood $\mathcal U$ of $\E$ with $\sup_{z\in \mathcal U\cap (\overline{\C}\setminus \E)}(h-\cG_\E)\leq M+1$.
    Since the difference is harmonic on $\overline{\C}\setminus \mathcal U$, we conclude that
    $\sup_{z\in \overline\C\setminus \E}(h(z)-\cG_\E(z,\bC))<\infty$. Thus, by the maximum principle and the reverse inequality, $h=\cG_\E$ on $\overline{\C}\setminus \E$. Since the first sequence was arbitrary, we have \eqref{it:uniformlimitk}.

    Since the implication \eqref{it:planek}$\implies$\eqref{it:upperhalfplanek} is clear, we may conclude.
\end{proof}

We now put the subsequences together and use Corollary~\ref{cor:equivalence} to show that regular behavior occurs for one $k$ if and only if it happens for all.
\begin{proof}[Proof of Theorem~\ref{thm:CRegMain}]
Applying Lemma~\ref{lemma:CRegMain} for all $k$ implies equivalence of conditions (ii), (iv), (v), (vi), (vii) from Theorem~\ref{thm:CRegMain}.  By Corollary~\ref{cor:equivalence}, for some $z \in \bbC_+$, the condition
\[
\limsup_{j\to\infty} \frac 1{j(g+1)+k} \log \lvert \tau_{j(g+1)+k}(z) \rvert \le \cG_\E(z,\bC)
\]
holds for one value of $k$ if and only if it holds for all. Due to Lemma~\ref{lemma:CRegMain}, this immediately implies equivalence of conditions (i) and (iii) from Theorem~\ref{thm:CRegMain}. It remains to prove equivalence of (ii), (iii).

\eqref{it:CRegConstants}$\implies$ \eqref{it:CRegConstants2}: For $n \in \bbN$ and $1\le k \le g+1$, denote by $N(n,k)$ the integer such that $n+1 \le N(n,k) \le n+g+1$ and $N(n,k) - k$ is divisible by $g+1$. Then $N(n,k) / n \to 1$ as $n\to\infty$ so \eqref{it:CRegConstants} implies $\lim_{n\to\infty} \kappa_{N(n,k)}^{1/n} = \lambda_k^{1/(g+1)}$. Taking the product over $k=1,\dots,g+1$ gives \eqref{it:CRegConstants2}.

\eqref{it:CRegConstants2}$\implies$ \eqref{it:CRegConstants}: Similarly to the above, Theorem~\ref{thm12} shows that for all $k$,
\begin{equation}\label{18aug4}
\liminf_{n\to\infty} \kappa_{N(n,k)}^{1/n} \ge \lambda_k^{1/(g+1)}.
\end{equation}
Thus, if \eqref{it:CRegConstants} was false, this would mean that for some $k=m$, $\limsup_{n\to\infty} \kappa_{N(n,m)}^{1/n} > \lambda_m^{1/(g+1)}$. Taking products over $k$, we would have
\[
\limsup_{n\to\infty} \left( \prod_{k=1}^{g+1} \kappa_{N(n,k)} \right)^{1/n} \ge \limsup_{n\to\infty} \kappa_{N(n,m)}^{1/n} \liminf_{n\to\infty} \Biggl( \prod_{\substack{1 \le k \le g+1 \\ k \neq m}} \kappa_{N(n,k)} \Biggr)^{1/n} > \left( \prod_{k=1}^{g+1} \lambda_k \right)^{1/(g+1)}
\]
(the last step again uses \eqref{18aug4} for all $k\neq m$). This would contradict \eqref{it:CRegConstants2}, so the proof is complete.
\end{proof}

We now prove a seemingly special case of Corollary~\ref{thmST1}.

\begin{proposition}
Assume that the sequence  $\bC$ contains $\infty$. Then $\mu$ is Stahl--Totik regular if and only if it is $\bC$-regular.
\end{proposition}

\begin{proof}
	Let us assume that $\mu$ is $\bC$-regular and let $p_n$ denote the orthonormal polynomial with respect to $\mu$. Fix $z\in\bbC$. Since $\infty$ is in $\bC$,
$p_n\in\cL_{n(g+1)}$, so the orthonormal polynomials can be expressed in the basis of orthonormal rational functions as
\begin{align*}
p_n(z)=\sum_{m=0}^{n(g+1)}c_m\tau_m(z),\quad \sum_{m=0}^{n(g+1)}|c_m|^2=1.
\end{align*}
Thus, in particular, $|c_\ell|\leq 1$ and we get
\begin{equation} \label{19jul1}
|p_n(z)|\leq (1+n(g+1)) \sup_{0\leq m\le n(g+1)}|\tau_m(z)|.
\end{equation}

By Theorem~\ref{thm:CRegMain}, for q.e. $z\in\E$,
\begin{equation}\label{30nov5}
\limsup_{\ell \to\infty} \frac 1\ell \log \lvert \tau_\ell(z) \rvert \le 0.
\end{equation}
Thus, for q.e. $z\in\E$, \eqref{19jul1} implies
\begin{equation}\label{30nov1}
\limsup_{n \to \infty}\frac{1}{n}\log|p_n(z)|\leq 0.
\end{equation}
Thus, $\mu$ is Stahl--Totik regular.

Conversely, assume that $\mu$ is Stahl--Totik regular. For $n = j(g+1) + k$, the polynomial $R_n$ is a divisor of $R_{g+1}^{j+1}$, so we can write $\tau_n = \frac{P_n}{R_{g+1}^{j+1}}$ where $\deg P_n \le n+g$. For any $\epsilon > 0$ there exists a polynomial $Q_\epsilon$ such that $1 - \epsilon \le Q_\epsilon R_{g+1} \le 1+ \epsilon$ on $\E$. Thus,
\begin{equation}\label{30nov2}
\lvert \tau_n(z) \rvert \le (1-\epsilon)^{-j-1} \lvert P_n(z) Q_\epsilon^{j+1}(z) \rvert
\end{equation}
and $\lVert P_n Q_\epsilon^{j+1} \rVert \le (1+\epsilon)^{j+1}$ since $\tau_n$ is normalized.
Since $P_n Q_\epsilon^{j+1}$ is a polynomial of degree at most $n+g+ (j+1) \deg Q_\epsilon$, similarly to the above, representing it in the basis of polynomials shows
\begin{equation}\label{30nov3}
\lvert P_n(z) Q_\epsilon^{j+1}(z) \rvert  \le (1+\epsilon)^{j+1}(n+g+1+(j+1)\deg Q_\epsilon) \sup_{0 \le m \le n+g+(j+1)\deg Q_\epsilon}  \lvert p_n(z) \rvert.
\end{equation}
Since $n+g + (j+1) \deg Q_\epsilon = O(n)$ as $n\to\infty$, the supremum in \eqref{30nov3} grows subexponentially whenever \eqref{30nov1} holds. By \eqref{30nov2}, this implies
\[
\limsup_{n\to\infty} \frac 1n \log \lvert \tau_n(z) \rvert \le \frac{1}{g+1}\log\left(\frac{1+\epsilon}{1-\epsilon}\right).
\]
Since $\epsilon> 0$ is arbitrary, we conclude that \eqref{30nov1} implies \eqref{30nov5}, so \eqref{30nov5} holds q.e.\ on $\E$.
\end{proof}

From this seemingly special case, Theorem~\ref{thmST0}, and Corollary~\ref{thmST1} follow easily:
\begin{proof}[Proof of Theorem~\ref{thmST0}]
By applying a conformal transformation, the special case shows that $\mu$ is $\bC_1$-regular if and only if it is $(\bc_k)$-regular for any single $\bc_k$ in $\bC_1$. By applying this twice, we conclude that if $\bC_1$, $\bC_2$ have a common element, then $\mu$ is $\bC_1$-regular if and only if it is $\bC_2$-regular.

By applying that conclusion twice, we will finish the proof. Namely, for arbitrary $\bC_1$, $\bC_2$, choose a sequence $\bC_3$ which has common elements with both $\bC_1$ and $\bC_2$. Then $\mu$ is $\bC_1$-regular if and only if it is $\bC_3$-regular if and only if it is $\bC_2$-regular.
\end{proof}

\begin{proof}[Proof of Corollary~\ref{thmST1}]
The result follows by taking $\bC_2=(\infty)$ in Theorem~\ref{thmST0}.
\end{proof}

\begin{proof}[Proof of Theorem~\ref{thmST2}]
By Lemma~\ref{lemmaMobius},  $f_* \mu$ is Stahl--Totik regular if and only if $\mu$ is $(f^{-1}(\infty))$-regular, and by Corollary~\ref{thmST1}, this is equivalent to Stahl--Totik regularity of $\mu$.
\end{proof}

\begin{proof}[Proof of Theorem~\ref{thmDOSlimit}]
(a) We note that by Corollary~\ref{thmST1} we may use Theorem~\ref{thm:CRegMain}. Fix $1\leq k\leq g+1$, and use conformal invariance to assume $\bc_k=\infty$. Given a subsequence of $n(j)=j(g+1)+k$, we use precompactness to pass to a further subsequence $n_\ell=j_\ell(g+1)+k$ with $\wlim_{\ell\to\infty}\nu_{n_\ell}=\nu$.
We write
\begin{align}\label{eq:Greenfunction}\cG_\E(z,\bC)=\Phi_{\rho_{\E,\bf C}}(z)+\frac{1}{g+1}\log\lambda_k-\frac1{g+1}\sum_{\substack{m=1\\m\ne k}}^{g+1}\log|z-\bc_m|\end{align}
which we will use to show $\Phi_{\nu}=\Phi_{\rho_{\E,\bf C}}$.
By \eqref{it:CRegConstants}, we may apply Theorem~\ref{thm:June9} with $\alpha=\frac{1}{g+1}\log\lambda_k$. Then, \eqref{it:uniformlimit} yields $h=\cG_\E$ off the real line, and thus the equality between the representations \eqref{24jul1} and
\eqref{eq:Greenfunction}
gives $\Phi_{\nu}(z)=\Phi_{\rho_{\E,\bf C}}(z)$ on $\C\setminus \R$. By the weak identity principle, this equality extends to $\C$. Applying the distributional Laplacian to both sides gives $\nu=\rho_{\E,\bf C}$. Thus, $\wlim_{n\to\infty}\nu_n= \rho_{\E,\bf C}$.

(b) The main ingredient is a variant of Schnol's theorem; for any $n$, $\int \lvert \tau_n \rvert^2 \ d\mu = 1$, so
\[
\sum_{n=1}^\infty n^{-2} \int \lvert \tau_n \rvert^2 \ d\mu  < \infty.
\]
By Tonelli's theorem, it follows that $\sum_{n=1}^\infty  n^{-2} \lvert \tau_n \rvert^2 < \infty$ $\mu$-a.e., so there exists a Borel set $B\subset \bbC$ with $\mu(\bbC\setminus B)=0$ such that
\begin{equation}\label{8aug1}
\limsup_{n\to\infty} \frac 1n \log \lvert \tau_n(z) \rvert \le 0, \qquad \forall z \in B.
\end{equation}
Suppose $\mu$ is not regular. Then, by Theorem~\ref{thm:CRegMain} \eqref{it:CRegConstants}, there is a $1\leq k\leq g+1$ with
\[
\limsup_{j\to\infty}\frac{1}{n(j)}\log\kappa_{n(j)} > \frac{1}{g+1}\log \lambda_k.
\]
Using conformal invariance, we may assume $\bc_k=\infty$, and we can pass to a subsequence $n_\ell = j_\ell (g+1)+ k$ such that $\alpha := \lim_{\ell \to\infty}\frac{1}{n_\ell}\log \kappa_{n_\ell} > \frac{1}{g+1}\log \lambda_k $, where $\alpha\in \R\cup \{+\infty\}$ by Theorem~\ref{thm:June9} \eqref{it:impossible}.
Since $\wlim_{n\to \infty}\nu_n=\rho_{\E,\bC}$, by comparing \eqref{24jul1} and \eqref{eq:Greenfunction}, we have for $z\in \C\setminus \R$,
\begin{equation}\label{26aug1}
\lim_{\ell \to \infty}\frac{1}{n_\ell }\log \lvert \tau_{n_\ell}(z) \rvert =\cG_\E(z,\bC)+d
\end{equation}
where $d=\alpha-\frac{\log\lambda_k}{g+1} > 0$. By the upper envelope theorem applied to the sequence $\{\nu_{n_\ell}\}_{\ell \in\N}$, there exists a polar set $X$ such that \eqref{26aug1} also holds for all $z \in \bbC \setminus X$. Moreover, since $\cG_\E(z,\bC) \ge 0$ for all $z \in \bbC$, we conclude that
\[
\limsup_{n\to\infty} \frac 1n \log \lvert \tau_n(z) \rvert \ge \lim_{\ell \to \infty}\frac{1}{n_\ell }\log \lvert \tau_{n_\ell}(z) \rvert \geq d, \qquad \forall z \in \bbC \setminus X.
\]
Comparing with \eqref{8aug1} shows that $B \subset X$, so $\mu$ is supported on the polar set $X$.
\end{proof}

\begin{proof}[Proof of Theorem~\ref{thmRegularityGMP}]
Defining $n(j)=j(g+1)+k_\infty$ and using Lemma~\ref{claim:1} to compute a telescoping product,
\begin{equation}\label{16aug5}
\left( \prod_{\ell=1}^{j} \beta_\ell \right)^{1/j} = \left( \prod_{\ell=1}^j \frac{\kappa_{n(\ell)}}{\kappa_{n(\ell+1)}} \right)^{1/j} = \kappa_{n(1)}^{1/j}\kappa_{n(j+1)}^{-1/j}.
\end{equation}
The first term on the right-hand side is independent of $j$, so $\kappa_{n(1)}^{1/j} \to 1$ as $j\to \infty$. For the second factor, using Theorem~\ref{thm12} we compute
\[
\liminf_{j\to\infty}\kappa_{n(j+1)}^{1/j}\geq  \lambda_{k_\infty}
\]
and we have the upper bound \eqref{16aug3} for the $\limsup$ of  \eqref{16aug5}. Similarly, using the criterion Theorem~\ref{thm:CRegMain} (\ref{it:CRegConstants}), it follows from \eqref{16aug5} that $\mu$ is $\bC$-regular if and only if \eqref{16aug4} holds.
\end{proof}

\section{GMP matrices 2}\label{sec:CNCondition}
The proof of Theorem \ref{conj1} will rely heavily on the results of \cite{YuditskiiAdv}. In this section we will recall some properties of GMP matrices  from \cite{YuditskiiAdv} which we will use in the proof of Theorem \ref{conj1}. However, in order to justify the use of those constructions, we need to add some explanation of the structure of GMP matrices. This technical explanation is necessary in order to understand the action on Jacobi matrices caused by a single coefficient stripping step on GMP matrices; since such a step changes the location of $\infty$ in the sequence of poles, it links our GMP matrices which naturally arise from ORF expansions, and those in \cite{YuditskiiAdv}, which naturally arise from functional models of reflectionless operators. This link will allow us to use parts of the analysis of \cite{YuditskiiAdv}.

As noted in the beginning of Section \ref{sec:GMPandGrowth}, GMP matrices split up into blocks due to the appearance of some $c_{k_\infty}=\infty$. However, there is a choice whether to place the ``window'' of block size $(g+1)\times (g+1)$ so that $c_{k_\infty}$ is the last element of the previous block, or the first element of the next block. In this paper, the latter choice has been more natural  (i.e., to split before $\infty$), because it corresponds to the choice $\tau_0=1$ in the rational function construction. From now on, we will call this the RF structure. On the other hand, in \cite{YuditskiiAdv} the first choice was more natural (i.e., to split after $\infty$) for the functional model construction, and we will call this the FM structure. Alternatively, recall that we discussed in Remark \ref{rem:GMPstructure} that one could view the GMP structure also as overlapping $(g+2)\times(g+2)$ blocks. The RF-structure then corresponds to placing the $(g+1)\times(g+1)$ $B$ block at the upper left corner of the bigger block, whereas the FM structure corresponds to placing the $B$ blocks at the lower right corner. This is shown in the figure below, where the blue lines indicate a $B$ block corresponding to the RF structure and the red lines a $B$ block corresponding to the FM structure. Moreover, $\tilde p_0$ denotes the positive entry on the outermost diagonal:
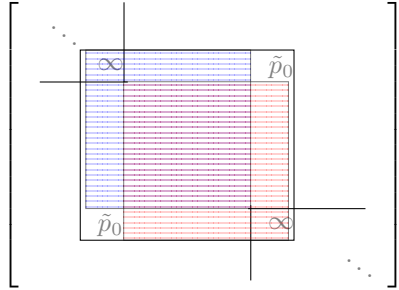
\begin{align*}
\begin{tikzpicture}[baseline=-0.5ex]
\definecolor{my_orange}{RGB}{243, 171, 0};  
\definecolor{my_green}{RGB}{181,230,29};
\matrix (m) [matrix of math nodes, nodes in empty cells, column sep=0mm, row sep=0mm, nodes={rectangle, 
	minimum size=1.2em, text depth=0.25ex,
	inner sep=0pt, outer sep=0pt,
	fill opacity=0.5, text opacity=1,
	anchor=center},
left delimiter={[},right delimiter={]},ampersand replacement=\&] {
	\ddots \& \& 			    \& 			    \&  \&          	      \&  \&   	\&			   \\
	\& \infty  \& 			  	\& 		\&          \&       		      \&  \tilde p_0\&		\&			   \\
	\&   \&  \& 	\&          \&\&  \&    \&  \\
	\&   \& 			    \& 	    \&          \&       		      \&  \&		\&			   \\
	\&   \& 			    \&               \&  \&       		      \&  \&	\&			\\
	\& \& 			    \& 		 \&  \&          	      \&  \&   	\&			  \\
	\&\tilde p_0 \& 			    \& 		 \&  \&          	      \&  \infty\&  	\&			   \\
	\& \& 			    \& 		 \&  \&          	      \&  \&  	\&		\ddots 	   \\	
} ;
\draw (-2.17,0.84) -- (-1,0.84);
\draw (-1.05,0.84) -- (-1.05,1.9);
\draw (0.635,-0.8) -- (0.635,-1.8);
\draw (0.6,-0.845) -- (2.16,-0.845);
\draw (m-7-2.south west) rectangle (m-2-7.north east);
\draw[pattern={Lines[angle=-45,distance={2pt},
	line width=0.5pt]},pattern color=blue,opacity=0.5] (m-6-2.south west) rectangle (m-2-6.north east);
\draw[pattern={Lines[angle=45,distance={2pt},
	line width=0.5pt]},pattern color=red,opacity=0.5] (m-7-3.south west) rectangle (m-3-7.north east);
\end{tikzpicture}
\end{align*}

The two structures can be translated into each other, by means of the formulas \eqref{eq:pqTranslation} presented below.
Moreover, we will show below that they are also linked by a coefficient stripping formula. 

For the reader's convenience, we recall the FM structure of GMP matrices as introduced in \cite{YuditskiiAdv}. Although the RF and FM structure are just a different interpretation of the same object, namely a GMP matrix, it will be convenient to have a separate notation. For a GMP matrix written in the FM structure we will use $A$, respectively for its blocks $A_k,B_k$, and we will use $\tilde A,\tilde A_k,\tilde B_k$, for GMP matrices written in the RF structure. Note that this is a change from the notation used in previous sections. 

Fix a finite sequence $\bC=(\bc_1,\dots, \bc_g)$ and recall that $X^-$ denotes the upper triangular part of a matrix $X$ (excluding the diagonal), and $X^+$ the lower triangular part (including the diagonal). Then we say that $A$ acting on $\ell^2(\bbZ)$ is GMP structured, and denote it by $A\in \A$, if it is a $(g+1)$-block Jacobi matrix
\begin{equation*}
A=\begin{bmatrix}
\ddots&\ddots&\ddots&& &\\
&A^*_{-1}&B_{-1}&A_0& & \\
& &A^*_{0}&B_{0}&A_1& \\
& & &\ddots&\ddots&\ddots
\end{bmatrix}
\end{equation*}
such that
\begin{equation*}
A_j=\delta_g \vp_j\!^*,
\quad
B_j
=(\vq_j \vp_j\!^*)^-+(\vp_j\vq_j\!^*)^++\hat\bC,\quad \vp_j,\vq_j\in\bbR^{g+1},
\end{equation*}
and
\begin{equation*}
\hat \bC=\begin{bmatrix}
\bc_1& & & \\
& \ddots& & \\
& & \bc_g & \\
& & &0
\end{bmatrix},\ 
\vp_j=
\begin{bmatrix}
p^{(j)}_0\\
\vdots\\
p^{(j)}_g
\end{bmatrix}, \ 
\vq_j=
\begin{bmatrix}
q^{(j)}_0\\
\vdots\\
q^{(j)}_g
\end{bmatrix}, \quad p^{(j)}_g>0.
\end{equation*}
We then say an operator $A\in  \A$ is a two-sided GMP matrix, if the resolvents  $(\bc_\ell - A)^{-1}$ exist for all $1\leq \ell \leq g$ and $S^{-\ell}(\bc_\ell - A)^{-1}S^\ell \in \A$. In this case we write $A\in\GMP(\bC).$ Again we call the generating coefficients $\{\vp_j,\vq_j\}_{j\in\Z}$ the GMP coefficients of $A$.

We encounter several differences compared to  the RF structure presented in Section \ref{sec:GMPandGrowth}. First of all the 0 in $\hat \bC$ is now in the last, rather than in the first, position. Moreover, in the definition of $A_j$, the vector $\vp_j$ is now a row vector in the last row, rather than a column vector in the first column. This is exactly due to shifting the position of $\infty$ as described above. Extending the structure of GMP matrices to two-sided operators on $\ell^2(\bbZ)$, it is not hard to see that the RF and the FM structure can be translated into each other, simply by shifting the window of size $(g+1)\times(g+1)$ by one. In this process the role of $p_j$ and $q_j$ changes. that is,  for $1\leq k \leq g$ we have 
\begin{align}\label{eq:pqTranslation}
\tilde p_k^{(j)}=q_{{k-1}}^{(j)}p_g^{(j)},\quad \tilde q_k^{(j)}=\frac{p_{k-1}^{(j)}}{p_g^{(j)}}.
\end{align}
More importantly for us is that the positive entries are the same, i.e., 
\[
p_g^{(j)}=\tilde p_0^{(j)}.
\]
Following \cite{YuditskiiAdv}, we explain how to associate to a given GMP matrix a Jacobi matrix. Let $e_j$ denote the standard basis vectors in $\ell^2(\bbZ)$; recall that $\{e_{-1},e_0\}$ forms a spectral basis for two-sided Jacobi matrices in the sense that $\{ J^n e_j \mid n\in\bbN_0, j=-1,0 \}$ is dense in $\ell^2(\bbZ)$. Define the matrix resolvent function by 
\begin{align*}
R^J(z)=\begin{bmatrix}
\langle (J-z)^{-1}e_{-1},e_{-1}\rangle& \langle (J-z)^{-1}e_{0},e_{-1}\rangle\\
\langle (J-z)^{-1}e_{-1},e_{0}\rangle&\langle (J-z)^{-1}e_{0},e_{0}\rangle
\end{bmatrix}.
\end{align*}
Let $\ell^2_+=\ell^2(\bbN_0)$ and $\ell^2_-=\ell^2(\bbZ)\ominus\ell^2_+$ and $\Pi_\pm$ denote the projection onto $\ell^2_\pm$. Define $J_\pm=\Pi_\pm J\Pi_\pm$  and define the half-line resolvent functions by 
\[
m_+^J(z)=\langle (J_+-z)^{-1}e_0,e_0\rangle,\quad m_-^J(z)=\langle (J_--z)^{-1}e_{-1},e_{-1}\rangle.
\]

Then, essentially due to the structure 
\begin{align*}
J=\begin{bmatrix}J_-& 0\\
0& J_+
\end{bmatrix}
+
a_0(\langle\cdot,e_{-1}\rangle  e_0+\langle\cdot, e_{0}\rangle e_{-1}),
\end{align*}
one can see that 
\begin{align}\label{eq:ResolventJacobi}
R^J(z)=\begin{bmatrix}
m_{-}^J(z)^{-1}& a_0\\
a_0& m_+^J(z)^{-1}
\end{bmatrix}^{-1};
\end{align}
cf. \cite[pg 758]{DamYudAdvances}.

For GMP matrices, we need to modify the spectral basis. Define 
\begin{align}\label{eq:tilee0}
\tilde e_0=\frac{1}{a_0}\Pi_+Ae_{-1},\quad a_0=\|\Pi_+Ae_{-1}\|,
\end{align}
with the natural embedding into $\ell^2(\bbZ)$. Note that
\[
a_0\tilde e_0^\intercal=\begin{bmatrix}
\dots & 0&|&p^{(0)}_0& p^{(0)}_1&\dots&p^{(0)}_g&0&\dots 
\end{bmatrix}.
\]
Then $\{e_{-1},\tilde e_0\}$ form a spectral basis for $A$ and similarly as for Jacobi matrices we have 
\begin{align*}
A=\begin{bmatrix}A_-& 0\\
0& A_+
\end{bmatrix}
+
a_0(\langle\cdot,e_{-1}\rangle \tilde e_0+\langle\cdot,\tilde e_{0}\rangle e_{-1}).
\end{align*}
This allows us to define 
\begin{align*}
R^A(z)=\begin{bmatrix}
\langle (A-z)^{-1}e_{-1},e_{-1}\rangle& \langle (A-z)^{-1}\tilde e_{0},e_{-1}\rangle\\
\langle (A-z)^{-1}e_{-1},\tilde e_{0}\rangle&\langle (A-z)^{-1}\tilde e_{0},\tilde e_{0}\rangle
\end{bmatrix},
\end{align*}
and
\begin{align}\label{eq:resolvent}
m^A_-(z)=\langle (A_--z)^{-1}e_{-1},e_{-1}\rangle,\quad m^A_+(z)=\langle (A_+-z)^{-1}\tilde e_{0},\tilde e_{0}\rangle
\end{align}
and find similar to the Jacobi case that
\begin{align*}
R^A(z)=\begin{bmatrix}
m^A_{-}(z)^{-1}& a_0\\
a_0& m^A_+(z)^{-1}
\end{bmatrix}^{-1}.
\end{align*}

For a given GMP matrix $A$, the associated Jacobi matrix is simply defined by setting the resolvent functions to be equal, i.e., 
\begin{align}\label{eq:sameResolvents}
R^J(z)=R^A(z).
\end{align}
Note that this defines $J$ uniquely. Due to the common vector $e_{-1}$, it follows that 
\begin{align}\label{eq:abCoeff}
b_{-1}=\langle Je_{-1},e_{-1}\rangle=\langle Ae_{-1},e_{-1}\rangle=p_{g}^{(-1)}q_{g}^{(-1)},\quad a_0=\|\Pi_+ Je_{{-1}}\|=\|\Pi_+ Ae_{-1}\|=\|\vp_0\|,
\end{align}
which explains in hindsight the definition of $a_0$ in  \eqref{eq:tilee0}. 

\subsection{Shifts on GMP matrices}
For a vector $x\in \ell^2(\bbZ)$, let $|$ denote the splitting of $\ell^2_-$ and $\ell^2_+$, i.e., we write $x^\intercal=\begin{bmatrix}
\dots & x_{-1}|x_0& \dots
\end{bmatrix}$. We chose the vector of poles in the following way  $\begin{bmatrix}
\dots &\infty|\bc_1& \dots & \bc_g& \infty& \bc_1& \dots
\end{bmatrix}$.
That is for $A_+=\Pi_+A\Pi_+$ the first pole is $\bc_1\in\bbR$. However, if we consider $\tilde A_+=\Pi_+SAS^{-1}\Pi_+$, where $Se_k=e_{k+1}$ denotes the right shift, then the first pole of $\tilde A_+$ is $\infty$. 
\begin{align*}
\tilde A_+=\begin{tikzpicture}[baseline=-0.5ex]
\matrix (m) [matrix of math nodes, nodes in empty cells, column sep=0mm, row sep=0mm, nodes={rectangle, 
	minimum size=1.2em, text depth=0.25ex,
	inner sep=0pt, outer sep=0pt,
	fill opacity=0.5, text opacity=1,
	anchor=center},
left delimiter={[},right delimiter={]},ampersand replacement=\&] {
	\infty \,\,\,\& \& 			    \& 			    \&		   \\
	\& \mathbf {c}_1 \& 			  	\& 		\&            \\
	\&   \&  \& 	\&           \\
	\&   \& 			   	  \&\,\,A_+     \&         	   \\
	\&   \& 			    \&            \& 		\\
} ;
\draw (-1.5,0.7) -- (1.5,0.7);
\draw (-0.8,-1.2) -- (-0.8,1.18);
\end{tikzpicture}
\end{align*}

The resolvent functions of $A_+$ and $\tilde A_+$ are related by a coefficient stripping formula:

\begin{lemma}
	Let $A\in\GMP(\bC)$, $A_+=\Pi_+A\Pi_+$, $\tilde e_0,a_0,b_{-1}$ as in \eqref{eq:tilee0}, \eqref{eq:abCoeff} and define $\tilde A_+=\Pi_+SAS^{-1}\Pi_+$. Then the resolvent functions
	\begin{align*}
	m_+(z)=\langle (A_+-z)^{-1}\tilde e_0,\tilde e_0\rangle,\quad \tilde m_+(z)=\langle (\tilde A_+-z)^{-1}e_0,e_0\rangle
	\end{align*}
	are related by the coefficient stripping formula
	\begin{align}\label{eq:coeffStripping}
	\tilde m_+(z)=\frac{1}{b_{-1}-z-a_0^2m_+(z)}.
	\end{align}
\end{lemma}
\begin{proof}
	Recall that $S_+$ denotes the right shift on $\ell^2_+$ and define  
	\[
	f_0=S_+\tilde e_0=\frac{1}{a_0}\begin{bmatrix}
	0& p_0^{(0)}& p_1^{(0)}& \dots &p_g^{(0)}& 0& \dots
	\end{bmatrix}.
	\]
	Then we have 
	\begin{align*}
	\tilde A_+=\begin{bmatrix}b_{-1}& 0\\
	0& A_+
	\end{bmatrix}
	+
	a_0(\langle\cdot,e_{0}\rangle  f_0+\langle\cdot, f_{0}\rangle e_{0}).
	\end{align*}
	Then as for Jacobi matrices this implies \eqref{eq:coeffStripping}; cf. \cite[Theorem 3.2.4]{SimonSzego}.
\end{proof}

This lemma has a very natural interpretation. As we discussed above, GMP matrices split into blocks where $c_{k_\infty}=\infty$ and then there is some choice if we place $\infty$ as the last or the first element of a block. However, this discussion is irrelevant for Jacobi matrices, where all $c_k\equiv \infty$. Thus, if we associate to $A$ a Jacobi matrix $J$ by \eqref{eq:sameResolvents} and define $J_+=\Pi_+J\Pi_+$ and $\tilde J_+=\Pi_+SJS^{-1}\Pi_+$ and the associated $m_+$, $\tilde m_+$, then \eqref{eq:coeffStripping} becomes the standard coefficient stripping for Jacobi matrices.

There is another natural shift on GMP matrices. Namely, since the shift $A^{(1)}=S^{-(g+1)}AS^{(g+1)}$ preserves the GMP structure, one can describe how the resolvent functions of $A$ and $A^{(1)}$ are related. This will be done by so-called elementary Blaschke-Potapov factors of the third kind with poles at $\bc_1,\dots,\bc_g,\infty$; cf. \cite{ArovDym,Pot}.  In the following it will be convenient to use the notation 
\[
\bp=(p,q),\quad \vbp=(\vp,\vq)
\]

\begin{definition}
	For $p,q,\bc\in\bbR$ 
	\begin{equation}\label{facto2}
	\fa(z,\bc;\bp)=I-\frac{1}{\bc-z}\begin{bmatrix}p\\q\end{bmatrix} \begin{bmatrix}p&q\end{bmatrix}\fj=\exp\left(-\frac{1}{\bc-z}\begin{bmatrix}p\\q\end{bmatrix} \begin{bmatrix}p&q\end{bmatrix}\fj\right),\quad
	\fj=
	\begin{bmatrix}
	0&-1\\
	1&0
	\end{bmatrix}.
	\end{equation}
	represents the so-called Blaschke-Potapov factor of the third kind with a real pole $\bc$. 
	If $\bc=\infty$ it is of the form
	\begin{equation}\label{thematrix}
	\fa(z;\bp)=\fa(z,\infty;\bp)=
	\begin{bmatrix}
	0&-{p}\\
	\frac 1{p}&\frac{z-pq}{p}
	\end{bmatrix}.
	\end{equation}
Define the matrix function
\begin{align}\label{eq:matrixFunction}
\fA(z,\vbp)=\begin{bmatrix}
\fA_{11}& \fA_{12}\\\fA_{21}&\fA_{22}
\end{bmatrix}(z,\vbp)=\fa(z,\bc_1;\bp_0)\dots\fa(z,\bc_g;\bp_{g-1})\fa(z;\bp_g).
\end{align}
\end{definition}
The important role of the function $\fA(z,\vbp)$ will become clear by the following theorem.
\begin{theorem}{\cite[Theorem 2.13 and Theorem 2.15]{YuditskiiAdv}}
	Let $A\in \GMP(\bC)$, $A^{(1)}=S^{-(g+1)}AS^{(g+1)}$ and $A_+$ and $A^{(1)}_+$ the projections onto $\ell^2_+$. Let $m^{A}_+$ and $m^{A^{(1)}}_+$ be the resolvent functions defined by \eqref{eq:resolvent}. Let $a_0^2=\|\vp_0\|^2,\quad (a_0^{(1)})^2=\|\vp_0^{\,(1)}\|^2$. Then 
	\begin{align}\label{eq:cfGMP}
	a_0^2m^{A}_+(z)=\frac{\fA_{11}(z,\vbp_0)((a_0^{(1)})^2m^{A^{(1)}}_+(z))+\fA_{12}(z,\vbp_0)}{\fA_{21}(z,\vbp_0)((a_0^{(1)})^2m^{A^{(1)}}_+(z))+\fA_{22}(z,\vbp_0)}.
	\end{align}
\end{theorem}

\subsection{Periodic GMP matrices}
We call a two-sided GMP matrix 1-periodic or simply periodic if $S^{g+1} A S^{-(g+1)} = A$. In this case $m^{A^{(1)}}_+=m^{A}_+$ and \eqref{eq:cfGMP} is a quadratic equation for $m^{A}_+$. This allows to describe the spectrum of $A$ in terms of the function $\fA(z)$.
\begin{theorem}{\cite[Theorem 1.8]{EichSIGMA16}}
	Let $A=A(\vbp)\in\GMP(\bC)$ be a periodic GMP matrix and $\fA(z,\vbp)$ be as in \eqref{eq:matrixFunction} and define the discriminant by 
	\[
	\Delta(z)=\tr \fA(z,\vbp).
	\]
	Then the spectrum of $A$ is a finite union of intervals, it is purely absolutely continuous and of multiplicity $2$ and it is given by
	\[
	\sigma(A)=\Delta^{-1}([-2,2])=\{z\in\bbC|\ \Delta(z)\in [-2,2]\}.
	\]
\end{theorem}

The inverse problem can also be answered explicitly. Namely, given a finite union of intervals $\E$, are there periodic GMP matrices with the given spectrum and if so can one describe the set of all such matrices? Crucially, the answer to both questions is positive for the special choice $\bC=\bC_\E$, where $\bC_\E$ denotes the zeros of the Ahlfors function associated to $\E$. 
We define the isospectral torus of periodic two-sided GMP matrices by
\begin{align*}
\cT_\E(\bC_\E)=\{\mathring A\in\GMP(\bC_\E), \mathring A \text{ is periodic and } \sigma(\mathring A)=\E\}.
\end{align*}
Henceforth, we will use $\mathring A$ for elements from $\cT_\E(\bC_\E)$. We point out that for arbitrary finite gap sets, the isospectral torus of Jacobi matrices usually consists of almost periodic operators, whereas for GMP matrices we can always work with periodic operators. This also makes it possible to characterize the isospectral torus by a magic formula for GMP matrices. Moreover, this then can be used to describe $\cT_\E(\bC_\E)$  also as an algebraic manifold.
Recall that $\L_n$ denotes the outermost positive entry of the resolvents $(\bc_\ell-A)^{-1}$. That is if $A\in\GMP(\bC)$  is a periodic GMP matrix let 
\begin{align*}
\L_\ell(A)&=\langle e_\ell, (\bc_{\ell+1}-A)^{-1}e_{\ell+g+1}\rangle\quad \text{ for }0\leq \ell\leq g-1,  \\
\L_{g}(A)&=\langle e_{g}, Ae_{2g+1}\rangle.
\end{align*}

The resolvent entries can again be given explicitly in terms of $\Delta(z)$. 
\begin{lemma}\cite[Theorem 2.17]{YuditskiiAdv}
	Let $A\in\GMP(\bC)$  be a periodic GMP matrix. Then for $0\leq \ell\leq g-1$
	\begin{align}\label{eq:ResolPeriodic}
	\L_\ell(A)=-(\Res_{\bc_{\ell+1}}\Delta)^{-1}=-\bigg(\tr\bigg(\prod_{k=0}^{\ell-1}\fa(\bc_{\ell+1},\bc_{k+1};\bp_{k})\begin{bmatrix}p_{\ell}\\q_{\ell}\end{bmatrix} \begin{bmatrix}p_{\ell}&q_{\ell}\end{bmatrix}\fj\prod_{k=\ell+1}^{g-1}\fa(\bc_{\ell+1},\bc_{k+1};\bp_{k})\fa(\bc_{\ell+1},\bp_g)\bigg)\bigg)^{-1}
	\end{align}
\end{lemma}

This allows to describe $\cT_\E(\bC_\E)$ as an algebraic manifold. Let us fix a finite union of $g+1$ intervals and let $\Delta_\E$ denote the associated discriminant defined in terms of the Ahlfors function \eqref{eq:discrRational}. Then for coefficients $\vbp$ let $A(\vbp)\in \GMP(\bC_\E)$ be a periodic GMP matrix and define 
\begin{align*}
f_0(\vbp)&=\l_{g+1}\langle\vp,\vq\rangle+d,\\
f_\ell(\vbp)&=\L_{\ell-1}(A(\vbp))\l_\ell-1,\quad\quad  \text{for } 1\leq \ell\leq g+1.
\end{align*}
and $\bF_\E:U\subset\bbR^{2(g+1)}\to \R^{g+2}$ by
\begin{align}\label{eq:DefAlgebraic}
\bF_\E(\vbp)=\begin{pmatrix}
f_0(\vbp),
&\dots&f_{g+1}(\vbp)
\end{pmatrix}.
\end{align}
We then define the isospectral manifold by
\[
\cI\cS_\E=\{\vbp\in\bbR^{2g}: \bF_\E(\vbp)=0\}.
\]
The name is justified by the following theorem:
\begin{theorem}{\cite[Theorem 1.6 and Theorem 1.10]{EichSIGMA16}}
	Let $A\in\GMP(\bC_\E)$, then
	\begin{align}\label{eq:magicFormula1}
	A\in \cT_\E(\bC_\E)\iff \Delta_\E(A)=S^{g+1}+S^{-(g+1)}.
	\end{align}
	Moreover, for $\vbp$ such that $A(\vbp)\in \GMP(\bC_\E)$ we have that 
	\begin{align*}
	A(\vbp)\in \cT_\E(\bC_\E)\iff \bF_\E(\vbp)=0.
	\end{align*}
\end{theorem}

\subsection{Resolvents in the general case and the Jacobi flow}
Similar to \eqref{eq:ResolPeriodic} one can also find explicit expressions for $\L_n$ for general (not necessarily periodic) GMP matrices. Let $A\in \GMP(\bC)$ and for $n=j(g+1)+\ell$ for $j\in\bbZ$ and $0\leq \ell\leq g$ set
\begin{align*}
	\L_n(A)=\begin{cases}
	\langle e_n,(\bc_{\ell+1}-A)^{-1}e_{n+g+1}\rangle,\quad \ell\neq g,\\
	\langle e_n,Ae_{n+g+1}\rangle.
	\end{cases}
\end{align*}
\begin{lemma}{\cite[Lemma 3.2]{YuditskiiAdv}}
	Let $A\in \GMP(\bC)$. Then for $n=j(g+1)+\ell$ and $\ell\neq g$ we have 
	\begin{align}\label{eq:Resolvent}
	\L_n(A)=-\bigg(\tr\bigg(\prod_{k=0}^{\ell-1}\fa(\bc_{\ell+1},\bc_{k+1};\bp_{k}^{(j+1)})\begin{bmatrix}p^{(j+1)}_{\ell}\\q^{(j+1)}_{\ell}\end{bmatrix} \begin{bmatrix}p^{(j)}_{\ell}&q^{(j)}_{\ell}\end{bmatrix}\fj\prod_{k=\ell+1}^{g-1}\fa(\bc_{\ell+1},\bc_{k+1};\bp^{(j)}_{k})\fa(\bc_{\ell+1},\bp^{(j)}_g)\bigg)\bigg)^{-1}
	\end{align}
\end{lemma}

The explicit representation will be crucial in the following. Moreover, let us mention that due to the finite band block structure of GMP matrices, building (formal) resolvents is a purely local computation (compare e.g. \cite[eq (3.8)]{YuditskiiAdv}). This can be seen by the formula above, where only the entries of $A$ from the blocks $j$ and $j+1$ are needed to compute $\L_n$. 

Recall that we discussed already in the beginning of this section that to any GMP matrix $A$ we can associate a Jacobi matrix $J$, namely by setting the resolvent functions equal \eqref{eq:sameResolvents}. Let us denote this map by $\cF$. It is a deep result from \cite[Proposition 5.5.]{YuditskiiAdv} that this map is (up to a certain identification) invertible. An important question is if we can express the Jacobi parameters of $J=\cF A$ in terms of the coefficients of $A$. 
Let $\{\vbp_j\}$ denote the GMP coefficients and $\{a_j,b_j\}$ the Jacobi coefficients. Then we have already seen that
\[
a_0=\|\vp_0\|,\quad b_{-1}=q_g^{(-1)}p_g^{(-1)}.
\]
Let 
\[
\mathcal S J=S^{-1} J S.
\]
and note that
\[
a_0(\mathcal S J)=a_1(J),\quad b_{-1}(\mathcal S J)=b_0(J),
\]
where by $a_k(J), b_k(J)$ we mean the Jacobi parameters of the Jacobi matrix $J$. 
Thus, if one understands the transform on GMP matrices which is induced by the shift action on Jacobi matrices, one can inductively obtain the Jacobi parameters by the formulas above. This leads to the definition of the Jacobi flow on GMP matrices, which is defined by the following commutative diagram:
\begin{equation}\label{defjfg+}
\begin{array}{ccc}
\text{GMP} & \xrightarrow{\mathcal J}  & \text{GMP}  \\
&   &    \\
_{\mathcal F}  \big\downarrow &   &  _{\mathcal F} \big\downarrow  \\
&   &    \\
\text{Jacobi}  &  \xrightarrow{\mathcal S}  &  \text{Jacobi}
\end{array}
\end{equation}
Let us mention that this is one of the reasons why it is convenient to work with two-sided operators. If in this construction we considered the shift action on $\ell^2_+$, which is not unitary, then it is possible that for some $m$, $\bc_k\in\sigma((S_+^*)^mJ_+S_+^m)$, and thus the corresponding half-line GMP matrix would not be well defined.

The Jacobi flow is defined and discussed in \cite[Section 4]{YuditskiiAdv}. We provide the motivating ideas of the Jacobi flow and its precise definition below. First, note that in \cite{YuditskiiAdv}, we have the ordering of the poles 
\[
\bC_A:=\left[\begin{array}{cc|ccccc:cc}
\dots & \infty&\bc_1& \bc_2&\dots&\bc_g&\infty&\bc_1&\dots 
\end{array}\right]
\]
and recall that we anchored the blocks between $\infty$ (at position $-1$) and $\bc_1$ (at position $0$). Note that for Jacobi matrices, all poles are equal to $\infty$, and $\cS J$ corresponds to shifting an $\infty$-pole from position $0$ to position $-1$. Now applying the spacial shift to GMP matrices would be of a different flavor, as it shifts $\bc_1$ from $0$ to $-1$. Thus, one first has to shift $\infty$, which is now at position $g+1$ to the front, and then one may apply the spacial shift. This is done in $g$-steps. The $\cO$ transform defined below corresponds changing the order from $\bC_A$ to 
\[
\tilde\bC_A=\left[\begin{array}{cc|cccccc:cc}
\dots & \bc_g&\bc_1& \bc_2&\dots&\bc_{g-1}&\infty&\bc_g& \bc_1\dots
\end{array}\right].
\]

Letting
\[
\bo(\phi)=\begin{bmatrix}
\sin \phi& \cos \phi\\\cos\phi& -\sin\phi
\end{bmatrix}
\]
we make the following definition.
\begin{definition}
	We define the map:
	\[
	\cO:\GMP(\bc_1,\bc_2,\dots,\bc_g)\to \GMP(\bc_g,\bc_1,\dots,\bc_{g-1})
	\]
	in the following way. Let $O=O_A$ be the block-diagonal matrix
	\[
	O=\begin{bmatrix}
	\ddots&&&\\
	&O_{-1}&&\\
	&&O_0&\\
	&&&\ddots
	\end{bmatrix}
	\]
	where $O_k$ are the $(g+1)\times(g+1)$ orthogonal matrices
	\begin{align}\label{eq:sinDef}
	O_k=\begin{bmatrix}
	I_{g-2}&0\\
	0& \bo(\phi_k)
	\end{bmatrix},\quad\begin{bmatrix}
	\sin\phi_k&\cos\phi_k
	\end{bmatrix}
	=\frac{\begin{bmatrix}
		p_{g-1}^{(k)}&p_{g}^{(k)}
		\end{bmatrix}}{\sqrt{(p_{g-1}^{(k)})^2+(p_{g}^{(k)})^2}}.
	\end{align}
	Then
	\begin{align}\label{def:Otransform}
	\cO A:=SO_A^*AO_AS^{-1}.
	\end{align}
\end{definition}
As explained above, the Jacobi flow then is defined by applying $\cO$ $g$-times, in order to shift $\infty$ through the full block. This leads to the following definition:
\begin{definition}
	We define the Jacobi flow transform 
	\[
	\cJ:\GMP(\bC_A)\to \GMP(\bC_A)
	\]
	by
	\[
	\cJ A:= S^{-(g+1)}\cO^{\circ g}AS^{g+1}.
	\]
\end{definition}
It is shown in \cite[Equation (4.8) and Lemma 4.4]{YuditskiiAdv} that there exists a block-diagonal unitary mapping $U_A$,  such that
\begin{align}\label{eq:JacobiFlow}
\cJ A=S^{-1}U_A^*AU_AS.
\end{align}

Let us also note that 
\begin{align}\label{eq:commJacandO}
S^{-(g+1)}\cO(A)S^{(g+1)}=\cO(S^{-(g+1)}AS^{(g+1)}),
\end{align}
which has the consequences
\begin{align}\label{eq:commJacandShift}
\cO(\cJ^{\circ m}A)=\cJ^{\circ m}(\cO A)
\end{align}
and
\[
S^{-(g+1)}(\cJ^{\circ m}A)S^{(g+1)}=\cJ^{\circ m}(S^{-(g+1)}AS^{(g+1)}).
\]
\section{Proof of Theorem \ref{conj1}}
The following lemma allows us to extend $J_+$ to a two-sided Jacobi matrix $J$ acting on $\ell^2(\bbZ)$ in a way such that $\bc_k$ belong to the resolvent domain of $J$. 
\begin{lemma}\label{lemmarank1}
	Let $\mu$ be a compactly supported probability measure such that $\E=\esssupp\mu$ is a union of $g+1$ intervals. Let
	\[
	m_+(z)=\int_{-\infty}^\infty\frac{1}{x-z}d\mu(x)
	\]
	and $J_+$ the associated Jacobi matrix and let $\bc_k\in \bbR\setminus\E$ for $1\leq k\leq g$. Then there exists a two-sided Jacobi matrix $J$ with the following properties:
	\begin{enumerate}[(i)]
		\item $J_+=\Pi_+J\Pi_+$;
		\item there exists $\mathring{J}\in\cT(\E)$ so that  $J_-:=\Pi_-S{J}S^{-1}\Pi_-$  obeys $J_- =\Pi_-S\mathring{J}S^{-1}\Pi_-$; 
		\item\label{it:3Jan13} $\bc_k$ belong to the resolvent domain of $J$;
		\item\label{it:4Jan13} $\bc_k$ belong to the resolvent domain of $\tilde J_+ := \Pi_+SJS^{-1}\Pi_+$. 
	\end{enumerate}
\end{lemma}
\begin{proof}
	Let $J$ denote the extended two-sided matrix. Note that $J$ is defined by $J_+$, $a_0, a_{-1}, b_{-1}$ and $J_-$. We fix $J_+$ and $a_0,a_{-1}$ and choose $J_-$ and $b_{-1}$ appropriately. 
	
	By \eqref{eq:coeffStripping} we have 
	\begin{align*}
	\tilde m_+(z)=\frac{1}{b_{-1}-z-a_0^2m_+(z)}.
	\end{align*}
	Thus, $\bc_k$ is a pole of $\tilde m_+(z)$ if and only if it is a zero of $b_{-1}-z-a_0^2m_+(z)$. Choose $b_{-1}$ so that 
	\[
	b_{-1}-\bc_k-a_0^2m_+(\bc_k)\neq 0.
	\] 
	This already defines $\tilde J_+$.
	
	Let us write \eqref{eq:ResolventJacobi} at position $-1$ rather then at position $0$ and let $m_-$ be the resolvent function of $J_-$ and $\tilde m_+$ the resolvent function of $\tilde J_+$. Then we see that 
	\[
	-\frac{1}{R_{-2-2}(z)}=-\frac{1}{m_-(z)}+a_0^2\tilde m_+(z),\quad -\frac{1}{R_{-1,-1}(z)}=-\frac{1}{\tilde m_+(z)}+a_0^2m_-(z).
	\]
	If $\tilde m_+(\bc_k)\in \{0,\infty\}$, we choose $m_-$ so that $m_-(\bc_k)\notin \{0,\infty\}$ and if $\tilde m_+(\bc_k)\notin \{0,\infty\}$ we set $m_-(\bc_k)=0$. In both cases $R_{-2,-2}(\bc_k)\neq \infty$ and $R_{-1,-1}(\bc_k)\neq \infty$ and we obtain \eqref{it:3Jan13}.
\end{proof}
We will apply Lemma~\ref{lemmarank1} in the following way. First we choose $\bc_1,\dots \bc_g$ as the zeros of the Ahlfors function of $\overline{\bbC}\setminus\E$.  Let $\mu$ be a given Stahl-Totik regular measure and $\E=\esssupp\mu$. To this measure we construct $J$ as above. Let further be $\tilde\mu$ be the spectral measure of $\tilde J_+$. Clearly $\E=\esssupp\tilde \mu$ and from the characterization of regularity by existence of the limit and equality in \eqref{6aug2} it follows that also $\tilde\mu$ is regular. Due to \eqref{it:3Jan13} we can form orthogonal rational functions with respect to the periodic sequence $\bC=(\bc_1,\dots,\bc_g,\infty,\bc_1,\bc_2,\dots)$. On the other hand \eqref{it:4Jan13} allows us to associate to $J$ a two-sided GMP matrix in the sense of \cite{YuditskiiAdv}. In particular, $\tilde J_+$ satisfies the assumptions of Lemma \ref{lem:BlockJacobiRegular}.

It was noted in \cite[Section 2.2]{YuditskiiAdv} that
\begin{equation}\label{13aug2}
- \log \lvert \Psi(z) \rvert = \sum_{k=1}^{g+1} G_\E(z,\bc_k)
\end{equation}
and that the Yuditskii discriminant has the form \eqref{eq:discrRational} for some $\lambda_k > 0$ and $d\in \bbR$. Note that the constants $\lambda_k$ can be found by computing the residue of $\Delta_\E$ at the poles $\bc_k$. By using \eqref{13aug1} and \eqref{13aug2}, we find the residues to be the same constants $\lambda_k$ defined in a more general setting in \eqref{6aug5}.

\begin{proof}[Proof of Lemma \ref{lem:BlockJacobiRegular}]
	Denote by $\mu$ the canonical spectral measure for $J$. Note that
	\[
	\sigma_\ess(A) = \esssupp \mu = \E = \Delta_\E^{-1}([-2,2]).
	\]
	Since $\Delta_\E$ maps $\bbR \setminus \{\bc_1,\dots,\bc_{g}\}$ to $\bbR$ and is piecewise strictly monotone, by a spectral mapping theorem, this implies that for $\bJ = \Delta_\E(A)$, $\sigma_\ess(\bJ) = [-2,2]$.
	
	As noted in the introduction, regularity of the Jacobi matrix $J$ implies $\bC_\E$-regularity by Corollary~\ref{thmST1}, and this can be characterized in terms of GMP matrix coefficients by Theorem~\ref{thmRegularityGMP}. The GMP matrix structure together with \eqref{eq:discrRational} implies that $\bJ = \Delta_\E(A)$ is a type 3 block Jacobi matrix \eqref{16aug7}; the diagonal entries of the off-diagonal blocks $\fv_j$ are given by $\lambda_k \Lambda_{j(g+1)+k}$ for $k=0,\dots,g$, with the convention $\lambda_0 = \lambda_{g+1}$. Thus,
	\[
	\det \fv_j = \prod_{k=0}^{g} \lambda_k \Lambda_{j(g+1)+k}.
	\]
	By applying the criterion for regularity in Theorem~\ref{thmRegularityGMP} to the GMP matrix $A$ and to its resolvents $(\bc_k - A)^{-1}$, we conclude that $\bJ$ obeys \eqref{13aug5}. It follows that $\bJ$ is regular with $\sigma_\ess(\bJ) =[-2,2]$.
\end{proof}

If $\tilde A_+$ is such that $\sigma_{\ess}(\tilde A_+)=\E$ and the corresponding measure is regular on $\E$, then $\Delta_\E(\tilde A_+)$ is a block Jacobi matrix which due to Lemma \ref{lem:BlockJacobiRegular} is regular for $[-2,2]$. Therefore, if $\{\fv_\ell,\fw_\ell\}$ denote the block Jacobi coefficients of $\Delta_\E(\tilde A_+),$ by \cite[Theorem 3.1]{SimonJAT09} we have
\begin{align}\label{eq:BlockCesaroS}
\lim_{N\to\infty}\frac{1}{N}\sum_{\ell=1}^{N}\|\fv_\ell-I\|+\|\fw_\ell\|=0.
\end{align}
We note that since
$
C=\sup_{\ell}(\|\fv_\ell(A)-I\|+\|\fw_\ell\|)<\infty,
$
it follows from Cauchy-Schwarz and the AM-GM inequality that
\begin{align*}
\left(\frac{1}{N}\sum_{\ell=1}^{N}\|\fv_\ell-I\|+\|\fw_\ell\|\right)^2\leq \frac{2}{N}\sum_{\ell=1}^{N}\|\fv_\ell-I\|^2+\|\fw_\ell\|^2\leq 2C\frac{1}{N}\sum_{\ell=1}^{N}\|\fv_\ell-I\|+\|\fw_\ell\|
\end{align*}
and thus
\begin{align}\label{eq:CNDiscriminant}
\lim_{N\to\infty}\frac{1}{N}\sum_{\ell=1}^{N}\|\fw_\ell\|^2+\|\fv_{\ell}-I\|^2=0\iff \lim_{N\to\infty}\frac{1}{N}\sum_{\ell=1}^{N}\|\fw_\ell\|+\|\fv_{\ell}-I\|=0.
\end{align}
We will use this equivalence freely in the following.

In the setting of periodic Jacobi matrices and polynomial discriminants (i.e., $\Delta$ is a polynomial and $\{\fv_\ell,\fw_\ell\}$ are the coefficients of the block Jacobi matrix $\Delta(J_+)$) it is shown in \cite{DamKillipSimAnnals} that
\begin{align}\label{eq:convISTorus}
\sum_{\ell=1}^{\infty}\|\fw_\ell\|^2+\|\fv_{\ell}-I\|^2<\infty\iff\sum_{m=1}^\infty d( (S_+^*)^m J S_+^m, \cT_\E^+)^2 <\infty.
\end{align}
It was then stated in \cite{SimonJAT09} that since all the arguments in \cite{DamKillipSimAnnals} are local, in this setting \eqref{eq:CNDiscriminant} yields \eqref{24jun2}. Let us emphasize that finite gap sets whose isospectral torus consists of periodic Jacobi matrices are very special and the arguments in \cite{SimonJAT09} only apply to this setting. Yuditskii \cite{YuditskiiAdv} has extended the work of \cite{DamKillipSimAnnals} and one has the same localness, but since the construction is quite involved, we will provide the main ideas of proof. In this case, the condition on the right-hand side of \eqref{eq:convISTorus} is still the same, i.e., a condition for a Jacobi matrix $J_+$, but on the left-hand side $\{\fv_\ell,\fw_\ell\}$ are the coefficients of the block Jacobi matrix $\Delta_\E(A)$, where $A$ is an associated GMP matrix and $\Delta_\E$ is the rational function as defined in \eqref{eq:discrRational}.

We will start with the main ingredients of the proof that the left-hand side in \eqref{eq:convISTorus} implies the right-hand side and mention certain modifications to our setting. After this preparatory work will show how this can be applied to our setting.

We concluded from regularity that $\bJ=\Delta_\E(\tilde A_+)$ satisfies \eqref{eq:BlockCesaroS}. As may be seen in \cite{DamKillipSimAnnals}, and \cite{YuditskiiAdv}, it is convenient to rewrite this condition into a ``multiplicative form''. This leads to the notion of the Killip-Simon functional that we will define below. For a GMP matrix $A\in \GMP(\bC_\E)$, we define the functional as in \cite[Section 6]{YuditskiiAdv} by
\begin{align}\label{eq:KSfunctional}
H_+(A)=\sum_{\ell=0}^\infty h(\fv_\ell,\fw_\ell,\fv_{\ell+1}),
\end{align}
where
\[
h(\fv_\ell,\fw_\ell,\fv_{\ell+1})=\frac{1}{2}\tr(\fv_\ell^*\fv_\ell+\fw_\ell^2+\fv_{\ell+1}\fv_{\ell+1}^*)-(g+1)-\log\det \fv_\ell\fv_{\ell+1}.
\]
For a square matrix $X$ its modulus is defined by $|X|:=\sqrt{X^*X}$. Moreover, define $G(|X|)=|X|^2-I-\log |X|^2$. Then we have
\[
2h(\fv_\ell,\fw_\ell,\fv_{\ell+1})=\tr\left( \fw_\ell^2+G(|\fv_\ell|)+G(|\fv_{\ell+1}^*|)\right).
\]
In particular, it follows from $|\fv_{\ell}|,|\fv_{\ell+1}|>0$ that $h(\fv_\ell,\fw_\ell,\fv_{\ell+1})>0$. 
In fact even more is true. There exists $\tilde C>1$ so that if $\|\fv_\ell-I\|<\frac{1}{2}$ then by \cite[Proposition 11.12]{DamKillipSimAnnals}
\[
\frac{1}{\tilde C}\|\fv_\ell-I\|\leq \||\fv_\ell|-I\|\leq \tilde C\|\fv_\ell-I\|.
\]
Thus, if $\tilde C\|\fv_\ell-I\|<\frac{1}{2}$  we conlude that $\||\fv_\ell|-I\|<\frac{1}{2}$ and thus the eigenvalues of $|\fv_\ell|$ are greater than $\frac{1}{2}$. Under this assumption (for $\ell$ and $\ell+1$) it is shown in \cite[Theorem 11.13]{DamKillipSimAnnals} that there exists a constant $C$ so that 
\begin{align}\label{eq:GHilbertEstimate}
\frac{1}{C}h(\fv_\ell,\fw_\ell,\fv_{\ell+1})\leq\left( \|\fv_\ell-I\|^2+\|\fw_\ell\|^2+\|\fv_{\ell+1}-I\|^2\right)\leq Ch(\fv_\ell,\fw_\ell,\fv_{\ell+1}).
\end{align}
A key observation is that the functional $H_+(A)$ is related to the shift action of $S^{g+1}$ on the GMP matrix $A$. But finally we want to conclude something about
$$\mathcal S J=S^* J S,$$
i.e., the shift action on $J$. This is another motivation of the Jacobi flow as defined above. 

The following key lemma, which follows essentially from \eqref{eq:JacobiFlow}, allows for the computation of the ``derivative'' in the Jacobi flow direction, and is essential in order to extract from the finiteness of $H_+(A)$ properties of the associated Jacobi matrix $J$.
\begin{lemma}{\cite[Lemma 6.1]{YuditskiiAdv}}\label{lem:JacobiFlowFunctional}
	Let $v_{jk}^{(\ell)},w_{jk}^{(\ell)}$ denote the matrix entries of $\fv_\ell,\fw_\ell$ and 
	\begin{align*}
	\delta_JH_+(A)=\frac{1}{2}\langle\Delta_\E(\cJ A)e_{-1},\Delta_\E(\cJ A)e_{-1}\rangle-1-\log(\cJ v)_{g,g}^{(-1)}(\cJ v)_{g,g}^{(0)}.
	\end{align*}
	Then
	\begin{align}\label{eq:JacobiDerivative}
	H_+(A)=H_+(\cJ A)+\delta_JH_+(A).
	\end{align}
\end{lemma}
\begin{proof}
	Using \eqref{eq:JacobiFlow} the proof is based on the realization that due to the diagonal structure of $U_A$, conjugating $A$ by $U_A$ does not affect $H_+$. Thus, $\delta_JH_+(A)$ corrects for the term which is omitted in $H_+(\cJ A)$ due to the shift. 
\end{proof}

For later reference let us mention that due to \eqref{def:Otransform} we can also relate $H_+(A)$ and $H_+(\tilde OA)$. Moreover, it is easy to see that we can also relate $H_+(A)$ and $H_+(S^{-(g+1)}AS^{(g+1)})$ explicitly. Moreover, this lemma allows to obtain from finiteness of $H_+(A)$, $\ell^2$ conditions for the coefficients of $\cJ^{\circ m}(A)$. We sketch the idea in the following. Let us define
	\begin{align*}
	\tilde H_+(A)=\sum_{m=0}^\infty\delta_JH_+(A(m)), \quad \text{ where } A(m)=\cJ^{\circ m}(A).
	\end{align*}
	Since all terms are positive, iterating \eqref{eq:JacobiDerivative} yields
	\begin{align*}
	\tilde H_+(A)\leq H_+(A).
	\end{align*}
	In particular, $H_+(A)<\infty$ implies $\tilde H_+(A)<\infty$. The vector $\Delta_\E(A(m))e_{-1}$ has only $2g+3$ non-vanishing entries which are entries of the last columns of $\fv_{-1}(m),\fw_{-1}(m)$ and $\fv_0(m)$. Let us denote this $2g+3$-dimensional vector by $x(m)$ and note that the first and the last component are the positive entries $x_0(m)=(v(m))_{g,g}^{(-1)}$ and $x_{2g+2}(m)=( v(m))_{g,g}^{(0)}$. With this notation we have 
	\begin{align}\label{eq:deltaJ}
	\delta_JH_+(A(m))=\frac{1}{2}\bigg(G(x_0(m))+G(x_{2g+2}(m))+\sum_{j=1}^{2g+1}x_j(m)^2\bigg).
	\end{align}
	Thus, $\tilde H_+(A)<\infty$ implies already $\ell^2$-conditions for the vector $x(m)$. This is used to conclude from $\tilde H_+(A)<\infty$ that $A(m)$ is $\ell^2$-close to be periodic and that the periodic operator is  $\ell^2$-close to $\mathcal I\mathcal S_\E$. That is, if $\{\vp_j(m),\vq_j(m)\}_{m\in\bbN_0}$ denote the GMP parameters of $A(m)$, then \cite[Theorem 1.20]{YuditskiiAdv}
\begin{equation}
\begin{aligned}\label{eq:theorem120}
\{\vbp_0(m)-\vbp_{-1}(m)\}_{m\in\bbN_0}\in\ell^2(\N_0,\bbR^{2(g+1)}),\\
\{\bF_\E(\vbp_0(m))\}_{m\in\bbN_0}\in\ell^2(\N_0,\bbR^{g+2}).
\end{aligned}
\end{equation}
To show how one obtains from  \eqref{eq:theorem120} convergence of $(S_+^{*})^mJS_+^m$ to $\cT_\E^+$ in the sense of \eqref{eq:convISTorus}, we need one more ingredient: it is well known that there are continuous functions, $\cA, \cB$, on $\bbR^g/\bbZ^g$, which can be expressed explicitly in terms of the Riemann theta function associated to $\E$\footnote{To be precise it is the Riemann theta function of the Riemann surface of the function $\sqrt{\prod_{k=0}^g(z-\ba_k)(z-\bb_k)}$, where ${\ba_k,\bb_k}$ denote the gap edges of $\E$ cf. \eqref{finitegapset}} \cite[Theorem 9.4.]{TeschlJacoobi}, and a fixed element $\chi\in \bbR^g/\bbZ^g$, such that
\begin{align}\label{eq:IsoParametrization}
\cT_\E=\{J(\a): \quad \a\in\bbR^g/\bbZ^g \}
\end{align}
and $J(\a)$ is the Jacobi matrix built from the coefficients
\begin{align}\label{eq:parameter}
a_m(\a)=\cA(\a-m\chi),\quad b_m(\a)=\cB(\a-m\chi).
\end{align}
Recall that by the definition of the Jacobi flow, if $J$ is the Jacobi matrix associated to $A$, then $S^{-m}JS^m$ is the Jacobi matrix associated to $A(m)$. Since every point of $\cI\cS_\E=\bF_\E^{-1}(0)$ is regular for $\bF_\E$, by \cite[Lemma 11.3]{DamKillipSimAnnals} there exists a constant $C>0$, such that  
	\begin{align}\label{eq:openGap}
	\dist(\vbp,\cI\cS_\E)\leq C\|\bF_\E(\vbp)\|,
	\end{align}
	where $\vbp$ are chosen from a fixed compact neighborhood of $\cI\cS_\E$, see also \cite[page 755]{YuditskiiAdv}.
Taking an element $\mathring A_m\in \cT_\E(\bC_\E)$ so that 
\[
\dist(\vbp(n),\cI\cS_\E)=\dist(A(\vbp(n)),\mathring A_m),
\]
one can conclude from \eqref{eq:theorem120} and \eqref{eq:openGap} that 
\begin{align}\label{eq:distance}
\sum_{n\geq 0}\dist(A(\vbp(n)),\mathring A_m)^2<\infty.
\end{align}
Letting $J(\a_m)\in\cT_\E$ be the Jacobi matrix with $F(\mathring A_m)=J(\a_m)$, then \eqref{eq:abCoeff} implies that 
\[
a(m)^2-\cA(\a_m)\in\ell^2_+,\quad b(m)-\cB( \a_m)\in\ell^2_+.
\]
Using in addition the smoothness of the Jacobi flow, one can show that 
\[
\a_n=\sum_{j=1}^n\vare^\a_m-m\chi,\quad \vare^\a_m\in\ell^2(\bbN_0,\bbR^{g}).
\]
This is even stronger than \eqref{eq:convISTorus}; cf. \cite[Lemma 7.2]{YuditskiiAdv}.

Before we start with our construction, we have to mention a certain technical issue. If $\{f_m\}$ is a sequence,  then clearly $\{f_m-1\}\in \ell^2$ implies 
	\[
	\liminf_{m\to\infty} f_m>0.
	\]
	If $|f_m-1|$ is only Ces\'aro summable, then this is not necessarily the case. However, for any $\d>0$ the set with $f_m<\d$ will be sparse in the following sense.
	Let us introduce the notation $\{f_m\}\in\CS$ for sequences $\{f_m\}$ satisfying
	\begin{align*}
	\lim\limits_{N\to\infty}\frac{1}{N}\sum_{m=1}^N|f_m|= 0,
	\end{align*}
	and we call a set $T \subset \bbN$ sparse if
	\[
	\lim_{N\to\infty} \frac{ \lvert  T \cap \{ 1, 2,\dots ,N\}  \rvert }N = 0.
	\]
	An elementary observation, which will be used repeatedly, is that for $f \in \CS$, the set $\{ m \in \bbN \mid \lvert f_m \rvert \ge \delta \}$ is sparse for any $\delta > 0$. This follows immediately from Markov's inequality.

We have already concluded from regularity that one and hence both of the conditions in \eqref{eq:CNDiscriminant} hold. Due to the phenomena described above and the $\log$ in the definition of $h(\fv_\ell,\fw_\ell,\fv_{\ell+1})$ it is not immediately clear that \eqref{eq:CNDiscriminant} also implies 
	\begin{align}\label{eq:KSfunction}
	\lim_{N\to\infty}\frac{1}{N}\sum_{\ell=0}^Nh(\fv_\ell,\fw_\ell,\fv_{\ell+1})=0.
	\end{align}
	However, using in addition once again regularity, we can show \eqref{eq:KSfunction}.
\begin{lemma}
	Let a Jacobi matrix satisfy the conditions of Lemma \ref{lem:BlockJacobiRegular} and $\{\fv_\ell,\fw_\ell\}$ denote the coefficients of the associated block Jacobi matrix $\bJ=\Delta_\E(A)$. Then \eqref{eq:KSfunction} holds.
\end{lemma}
\begin{proof}
	Recall that 
	\[
	h(\fv_\ell,\fw_\ell,\fv_{\ell+1})=\frac{1}{2}\tr\left((|\fv_{\ell}|^2-I)+(|\fv_{\ell+1}^*|^2-I)+\fw_\ell^2\right)-\log\det\fv_{\ell}\fv_{\ell+1}.
	\]
	Regularity allows us to consider the terms in $h(\fv_\ell,\fw_\ell,\fv_{\ell+1})$ separately. It follows directly from \eqref{eq:CNDiscriminant} that 
	\[
	\lim_{N\to \infty}\frac{1}{N}\sum_{\ell=0}^N\tr\fw_\ell^2=0.
	\]
	Moreover, \eqref{13aug5} implies that 
	\[
	\lim_{N\to \infty}\frac{1}{N}\sum_{\ell=0}^N\log\det\fv_\ell=0.
	\]
	Thus it remains to show that
	\begin{align}\label{eq:3Jan20223}
	\lim_{N\to \infty}\frac{1}{N}\sum_{\ell=0}^N\tr(|\fv_\ell|^2-I)=0,\quad \lim_{N\to \infty}\frac{1}{N}\sum_{\ell=0}^N\tr(|\fv_\ell^*|^2-I)=0.
	\end{align}
	For a matrix $A\in \Mat(n,\bbR)$, let $\sigma_i(A)$ denote its singular values and note that $\tr |A|=\sum \sigma_i(A)$. For $A,B\in \Mat(n,\bbR)$ we will need the following inequalities
	\begin{align}\nonumber
	|\tr A|&\leq \tr |A|,\\
	\sum_{i=1}^n\sigma_i(AB)&\leq \sum_{i=1}^n\sigma_i(A)\sigma_i(B),\label{eq:3Jan2022}
	\end{align}
	which can be found for instance in  \cite[eq. (3.3.35) and Theorem 3.3.14]{HornJohnMatrix}.
	Thus we have 
	\[
	\left|\frac{1}{N}\sum_{\ell=0}^N\tr(|\fv_\ell|^2-I)\right|\leq \frac{1}{N}\sum_{\ell=0}^N\tr(\left||\fv_\ell|^2-I\right|)=\frac{1}{N}\sum_{\ell=0}^N\tr(\left|(|\fv_\ell|-I)(|\fv_\ell|+I)\right|).
	\]
	Using \eqref{eq:3Jan2022} and a uniform bound on $\sigma_i(|\fv_\ell|+I)$ we get
	\begin{align*}
	\tr(\left|(|\fv_\ell|-I)(|\fv_\ell|+I)\right|)=\sum_{j=0}^g\sigma_j((|\fv_\ell|-I)(|\fv_\ell|+I))\leq C\sum_{j=0}^g\sigma_j(|\fv_\ell|-I),
	\end{align*}
	where $C$ does not depend on $\ell$.
	The last sum is the trace norm for $|\fv_\ell|-I$ and thus by the equivalence of norms on $\Mat(n,\bbR)$ we find $C_2$ so that 
	\[
	\left|\frac{1}{N}\sum_{\ell=0}^N\tr(|\fv_\ell|^2-I)\right|\leq C_2\frac{1}{N}\sum_{\ell=0}^N\||\fv_\ell|-I\|.
	\]
	Define the set
	\begin{align*}
	I_N=\left\{\ell:\ \tilde C\|\fv_\ell-I\|>\frac{1}{2} \right\}\cap[1,N]
	\end{align*}
	and note that \eqref{eq:BlockCesaroS} implies  
	\begin{align}\label{eq:3Jan20222}
	\lim\limits_{N\to\infty}\frac{|I_N|}{N}=0.
	\end{align}
	It follows as in \cite[Proposition 11.12]{DamKillipSimAnnals} that for $\ell\notin I_N$, there exists  a constant $C_3$ so that 
	\[
	\||\fv_\ell|-I\|\leq C_3\|\fv_\ell-I\|.
	\]
	For $\ell \in I_N$ we can estimate $\||\fv_\ell|-I\|$ uniformly and using \eqref{eq:3Jan20222} and \eqref{eq:CNDiscriminant} we obtain \eqref{eq:3Jan20223}. The proof for $\fv_\ell^*$ works the same by using \cite[Lemma 4.6.5.]{SimonSzego} instead of \cite[Proposition 11.12]{DamKillipSimAnnals}. This finishes the proof.
\end{proof}

We are now ready to adapt Yuditskii's construction \cite{YuditskiiAdv} to our setting. Let $\mu$ be a regular measure with $\esssupp \mu=\E$ and let $J_+$ be the associated Jacobi matrix. As already described after Lemma  \ref{lemmarank1}, we find $J$ and $\tilde J_+$ such that all $\bc_k\in\bC_E$ belong to the resolvent set of  $\tilde J_+$ and $J$ and $\tilde J_+$  is also regular. Let $\tilde A_+$ and $A$ denote the GMP matrix associated to $\tilde J_+$ and $J$ respectively and $\{\fv_\ell,\fw_\ell\}$ denote the block Jacobi coefficients of $\Delta_\E(A)$. Let us further truncate $A$ after $N$ positive blocks before $\infty$ (i.e. before the position $-1+N(g+1))$ and extend it by some element $\mathring{ A}\in\cT_\E(\bC_\E)$ so that $\bc_k\notin \sigma (A_N)$. To be precise, we first truncate $A$ and consider its resolvent function $a_0^2r_-$, then we can extend it as in Lemma \ref{lemmarank1} by some reflectionless $r_+$ so that all $\bc_k\in \bC_E$ belong to the resolvent set of the associated Jacobi matrix and then we consider the associated GMP matrix by \cite[Proposition 5.5]{YuditskiiAdv}. Since elements from the isospectral torus satisfy the magic formula and computing resolvents is a purely local process, we would like to conclude from the compactness of $\cT_{\E}(\bC_\E)$  that
\[
H_+(A_N)=\sum_{\ell=1}^{N}h(\fv_{\ell-1},\fw_\ell,\fv_\ell)+O(1)
\]
where $A_N$ denotes the truncation described above.
However, due to the log-term in the definition of $h(\fv_{\ell-1},\fw_\ell,\fv_\ell)$ one must be careful. At the place where we modify $A$ by extending it by $\mathring A$, by formula \eqref{eq:Resolvent}, when computing $\L_{n}$, in a certain range of $n$ given precisely below, one mixes coefficients from $A$ and $\mathring{ A}$. Thus we need to argue that 
\[
-\log\L_{n} 
\]
does not grow too fast so that we can still conclude that 
\begin{align}\label{eq:CSAN}
\lim\limits_{N\to\infty}\frac{1}{N}H_+(A_N)=0.
\end{align}
However, looking at the formula \eqref{eq:Resolvent} and the definition of the Blaschke-Potapov factors, if all the coefficients can be bounded uniformly, we see that if $p_g^{(j)}>\d$ we find a constant $C$ only depending on the bounds of the coefficients and of $\d$ so that 
\begin{align}\label{eq:Jan21}
\L_n(A_N)\geq C.
\end{align}
Note now that 
\[
\L_{-1+(N-1)(g+1)}(A_N)=p_g^{(N-1)},
\]
which is still a coefficient of $A$. But 
\[
\L_{-1+N(g+1)}(A_N)=:\mathring p_g
\]
is already a coefficient from $\mathring{ A}$. The mixing of coefficients of $A$ and $\mathring{ A}$ in computing $\L_n(A_N)$ happens for $-1+(N-1)(g+1)<n<-1+N(g+1)$. But in this case the only value that can make $\L_n(A_N)$  small is $\mathring p_g$, and for elements of the isospectral torus we know that 
\[
\mathring p_g=\frac{1}{\l_{g+1}}
\]
and thus we can conclude \eqref{eq:Jan21} and therefore \eqref{eq:CSAN}. 

Together with $\tilde H_+(A_N)\leq H_+(A_N)$, we conclude that 
\begin{align}\label{eq:tildeHplusAN}
\lim\limits_{N\to\infty}\frac{1}{N}\tilde H_+(A_N)=0.
\end{align}

Realizing that all the arguments in \cite[Theorem 1.20]{YuditskiiAdv} are local, using $2N$ blocks of $A$, we can obtain a local version of  this theorem.
\begin{proposition}\label{prop:thm120}
	Let $J$ be constructed as above and $A$ be the associated GMP matrix. Then, there exists an $N$ independent constant C and a sparse set $I_N$ such that
	\begin{equation}
	\begin{aligned}
	\sum_{m=1}^N\|\vbp_0(m)-\vbp_{-1}(m)\|^2\leq C(\tilde H_+(A_{2N})+|I_N|),\\
	\sum_{m=1}^N\|\bF_\E(\vbp_0(m))\|^2\leq C(\tilde H_+(A_{2N})+|I_N|)\label{eq:Jan14}.
	\end{aligned}
	\end{equation}
\end{proposition}
We will need a more quantitative version of \cite[Lemma 6.6]{YuditskiiAdv}:

\begin{lemma}\label{lem:ell2Matrix}
	Let $\psi_n$, $\tilde \psi_n$, $\tau_n$ and $\tilde \tau_n$ be given sequences and assume that there exists $\eta>0$ such that 
	\begin{equation}\label{eqsubl2}
	\cos\psi_n\ge\eta, \ \cos\tilde \psi_n\ge\eta, \ 0\leq\tau_n\leq \frac 1\eta, \ 0\leq\tilde \tau_n\leq \frac 1\eta.
	\end{equation}
	Define
	\begin{equation}\label{eqsub2}
	\alpha_n:=\begin{bmatrix}
	\tau_{n}&0\\
	0
	&1
	\end{bmatrix}
	\begin{bmatrix}
	\sin\psi_{n}&\cos\psi_{n}\\
	\cos\psi_{n}
	&-\sin\psi_{n}
	\end{bmatrix}-
	\begin{bmatrix}
	\sin\tilde\psi_{n}&\cos\tilde\psi_{n}\\
	\cos\tilde\psi_{n}
	&-\sin\tilde \psi_{n}
	\end{bmatrix}\begin{bmatrix}
	1&0\\0
	&\tilde\tau_{n}
	\end{bmatrix}.
	\end{equation}
	Then, there exists $C$ depending only on $\eta$ so that 
	\[
	\|\{\cos\psi_n-\cos\tilde\psi_n\}\|_{\ell^2(\bbN,\bbC)}\leq C\|\{\a_n\}\|_{\ell^2(\bbN,\bbC)^{2\times 2}},\quad \|\{\sin\psi_n-\sin\tilde\psi_n\}\|_{\ell^2(\bbN,\bbC)}\leq C\|\{\a_n\}\|_{\ell^2(\bbN,\bbC)^{2\times 2}}
	\]
\end{lemma}
\begin{proof}
	If $\|\{\a_n\}\|_{\ell^2(\bbN,\bbC)^{2\times 2}}=\infty$ the claim is trivial. If it is finite, set $S:=\|\{\a_n\}\|_{\ell^2(\bbN,\bbC)^{2\times 2}}$.
	The constant $C>0$ may increase throughout the proof. Directly from \eqref{eqsub2} we have 
	$$
	\|\{\cos\psi_n-\cos\tilde\psi_n\}\|_{\ell^2}\leq S\quad\text{ and }\quad
	\|\{\tau_n\cos\psi_n-\tilde \tau_n\cos\tilde\psi_n\}\|_{\ell^2}\leq S.
	$$
	Since
	\[
	\tau_n\cos\psi_n-\tilde\tau_n\cos\tilde\psi_n-\tilde\tau_n(\cos\psi_n-\cos\tilde\psi_n)=(\tau_n-\tilde\tau_n)\cos\psi_n,
	\]
	using   $\tilde \tau_n\leq\frac{1}{\eta}$ and $\cos\psi_n\geq \eta$ we find $C>0$ so that 
	\[
	\|\{\tau_n-\tilde\tau_n\}\|_{\ell^2}\leq CS.
	\]  
	Now, we have another two conditions
	$$
	\|\{\tau_n\sin\psi_n-\sin\tilde\psi_n\}\|_{\ell^2}\leq S\quad \text{and}\quad  
	\|\{\sin\psi_n-\tilde \tau_n\sin\tilde\psi_n\}\|_{\ell^2}\leq S.
	$$
	Using
	$$
	\sin\psi_n-\tilde \tau_n\sin\tilde\psi_n=\sin{\psi_n}-\tau_n\tilde \tau_n\sin\psi_n-\tilde\tau_n(\sin\tilde\psi_n-\tau_n\sin\psi_n)
	$$
	and $\|\{\tau_n\sin\psi_n-\sin\tilde\psi_n\}\|\leq S$ and $\tilde \tau_n\leq\frac{1}{\eta}$ we conclude that 
	\[
	\|\{\sin{\psi_n}(1-\tau_n\tilde \tau_n)\}\|_{\ell^2}\leq CS.
	\]
	Now we have
	\[
	1-\tau_n^2=1-\tau_n\tilde \tau_n+\tau_n(\tilde\tau_n-\tau_n)
	\]
	and since $|\sin{\psi_n}|\leq 1$ and $|\tau_n|\leq \frac{1}{\eta}$, we conclude 
	$$
	\|\{(\tau_n^2-1)\sin\psi_n\}\|\leq CS.
	$$
	Again by, $|\tau_n|\leq \frac{1}{\eta}$ we also get a bound for  $\{(\tau_n-1)\sin\psi_n\}$. Finally, since 
	$$
	\sin\psi_n-\sin\tilde\psi_n=\tau_n\sin\psi_n-\sin\tilde\psi_n-(\tau_n-1)\sin\psi_n,
	$$
	we obtain the also the estimate for $\{\sin\psi_n-\sin\tilde\psi_n\}$.
\end{proof}

\begin{proof}[Proof of Proposition \ref{prop:thm120}]
		In the proof we will find constants $C>0$ and sparse sets $I_N$. These quantities will change throughout the proof. Note that the union of sparse sets is clearly sparse. 
		First we mention an important locality property of the Jacobi flow. In the following we will derive estimates for entries of $A_{2N}(m)$ in the block $0$ and $-1$. Due to the locality property of the Jacobi flow, for $0<m\leq 2N-1$, the coefficients of $A_{2N}(m)$ and $A(m)$ coincide; this is nicely visualized in the diagram \cite[eq. (4.12) ]{YuditskiiAdv}. Similarly, we have already mentioned that computing entries of the resolvents, due to the band structure, can also be done locally. Thus, our estimates will be derived for the coefficients of $A_{2N}(m)$, but by restricting it to $0<m\leq N$ they agree with the coefficients associated to $A$. For this reason we will also notationally not distinguish between the coefficients of $A$ and the ones of $A_{2N}$.

		By the explanation following Lemma \ref{lem:JacobiFlowFunctional} and \eqref{eq:tildeHplusAN}, we conclude that 	
		\[
		\lim\limits_{N\to\infty}\frac{1}{N}\sum_{m=1}^N\bigg(G(x_0(m))+G(x_{2g+2}(m))+\sum_{j=1}^{2g+1}x_j(m)^2\bigg)=0.
		\]
		Notice that $G$ obeys
		\[
		c_\e^{-1} (x-1)^2\leq G(x)\leq c_\e(x-1)^2,\quad \forall x\in(\e,\e^{-1}).
		\]
		
		 Thus, we find a sparse set $I_N$ and $C>0$ so that 
		\begin{align}\label{eq:13Jan5}
		\begin{split}
		\sum_{m=1}^N&\bigg((x_0(m)-1)^2+(x_{2g+2}(m)-1)^2+\sum_{j=1}^{2g+1}x_j(m)^2\bigg)\\
		&\leq C\bigg(\sum_{m=1}^N\bigg(G(x_0(m))+G(x_{2g+2}(m))+\sum_{j=1}^{2g+1}x_j(m)^2)\bigg)+|I_N|)\bigg).
	\end{split}
		\end{align}
		Thus, for $1\leq j\leq 2g+1$,
		\[
		\|\{x_j(m)\}_{m=1}^N\|_2\leq C\left(\tilde H_+(A_N)+|I_N|\right)
		\]
		and
		\begin{align*}
		\|\{x_0(m)-1\}_{m=1}^N\|_2&\leq C\left(\tilde H_+(A_N)+|I_N|\right),\\ \|\{x_{2g+2}(m)-1\}_{m=1}^N\|_2&\leq C\left(\tilde H_+(A_N)+|I_N|\right).
		\end{align*}
		We note that 
		\[
		x_{2g+2}(m)=\l_0\L_{-1}(m)=\l_0p_g^{(0)}(m),
		\]
		and thus 
		\begin{align}\label{eq:13Jan2}
		\frac{1}{N}\sum_{m=1}^N(\l_0p_g^{(0)}(m)-1)^2=0.
		\end{align}
		This is one component of $F_\E$. 
		
		Let us now show the first inequality in \eqref{eq:Jan14}. Denote $\hat A=\cO A$, where $\cO A$ is the transform defined in \eqref{def:Otransform}. We use the hat  for all entries related to $\hat A$ and $\Delta_\E(\hat A)$, respectively. The entries of $A(m)$ are denoted by $\{p_k^{(j)}(m),q_k^{(j)}(m)\}$. Recall that $m$ corresponds to application of the Jacobi flow, $j$ denotes the block and $k$ the component of the vector $\vp_j(m)$. We use similar notation for $\hat A$, $\Delta_\E(A)$ and $\Delta_\E(\hat A)$. Due to the definition \eqref{def:Otransform}, we find
		\begin{equation}\label{eqar}
		\begin{bmatrix}
		v^{(0)}_{g-1,g-1}(m)& 0\\
		v^{(0)}_{g,g-1}(m)& \lambda_0 p^{(0)}_g(m)
		\end{bmatrix}
		\bo(\phi_{g}^{(0)}(m))
		=
		\bo(\phi_{g}^{(-1)}(m))\begin{bmatrix}
		\lambda_0\hat p^{(0)}_g(m)
		& 0\\
		\hat w^{(0)}_{0,g}(m)& \hat v^{(1)}_{0,0}(m)
		\end{bmatrix}.
		\end{equation}
		Note that $v^{(0)}_{g,g-1}(m)=x_{2g+1}(m)$. It was mentioned after Lemma \ref{lem:JacobiFlowFunctional} that $H_+(\hat A)$ can be expressed in terms of  $H_+(A)$. Therefore, we conclude by \eqref{eq:commJacandO} that \eqref{eq:13Jan2} also holds for $\hat p_g^{(0)}(m)$. Note that $x_{g+2}(m)= \hat w^{(-1)}_{0,g}(m)$. Since shifting by a full block in the very beginning only adds a fixed constant, and $\cJ$ commutes with this shift by \eqref{eq:commJacandShift}, we can apply Lemma \ref{lem:ell2Matrix} to \eqref{eqar} and obtain by \eqref{eq:13Jan5} that
		\[
		\|\{\sin\phi_g^{(-1)}(m)-\sin\phi_g^{(0)}(m)\}_{m=1}^N\|_{\ell^2}\leq C(\tilde H_+(A_N)+|I_N|).
		\]
		Thus, by \eqref{eq:sinDef}
		\begin{align}\label{eq:13Jan6}
		\|\{p_{g-1}^{(-1)}-p_{g-1}^{(0)}\}_{m=1}^N\|_{\ell^2}\leq C(\tilde H_+(A_N)+|I_N|).
		\end{align}
		Since by \cite[eq (4.2)]{YuditskiiAdv} one can pass from  $j$ to $j-1$ by using $\hat A$, we obtain  \eqref{eq:13Jan6} for $0\leq j\leq g$. Similarly, by \cite[eq (4.2)]{YuditskiiAdv}, one obtains the estimates for the $q_j$-coefficients. This finishes the proof of the first  inequality in \eqref{eq:Jan14}.
		
		It remains to prove \eqref{eq:Jan14} for the other components of $F_\E$. The proof of Lemma \ref{lem:ell2Matrix} yields an estimate for 
		$
		\|\{(v^{(-1)}_{g-1,g-1}(n)-1)\sin\phi^{(-1)}_g(n)\}\|_2
		$
		or, equivalently, it shows 
		\begin{equation*}\label{lambdal2}
		\|\{(\Lambda_{-2}(m)\lambda_g-1) p^{(-1)}_{g-1}(m)\}\|\leq  C(\tilde H_+(A_N)+|I_N|).
		\end{equation*}
		Since $p^{(-1)}_{g-1}(m)$ may approach to zero, it does not imply yet give an estimate for $\{(\Lambda_{-2}(m)\lambda_g-1)\}$. If we can also estimate 
		\begin{equation*}\label{lambdal21}
		\|\{(\Lambda_{-2}(m)\lambda_g-1) q^{(-1)}_{g-1}(m)\}_{m=1}^N\|,
		\end{equation*}
		then  $\inf_{m}\left((q^{(-1)}_{g-1}(m))^2+(p^{(-1)}_{g-1}(m))^2\right)>0$ yields 
		\[
		\|\{(\Lambda_{-2}(m)\lambda_g-1)\}_{m=1}^N\|\leq  C(\tilde H_+(A_N)+|I_N|).
		\]
		To this end, we note that 
		\begin{equation}\label{lambdal22}
		\Lambda_{-2}(m+1)=\frac{\cos\phi^{(-1)}_g(m)}{\cos\phi^{(-2)}_g(m)}\Lambda_{-2}(m).
		\end{equation}
		Indeed, by definition
		of the Jacobi flow
		$$
		U(\vp_{-2}(m))\begin{bmatrix} v^{(-2)}_{g,g}& & & \\
		*&v^{(-1)}_{0,0}& &\\
		*&* & \ddots &\\
		*&* &* &v^{(-1)}_{g-1,g-1}
		\end{bmatrix}(m+1)=\fv_{-1}(m) U(\vp_{-1}(m))
		$$
		the second from below entry in the last column in this matrix identity means exactly \eqref{lambdal22}. Since by the above, we can estimate $\|\{\cos\phi^{(-1)}_g(m)-\cos\phi^{(-2)}_g(m)\}_{m=1}^N\|_{\ell^2}$ we obtain 
		\[
		\|\{\Lambda_{-2}(m+1)-\Lambda_{-2}(m)\}_{m=1}^N\|_{\ell^2}\leq C(\tilde H_+(A_N)+|I_N|).
		\]
		Now by \cite[(4.10)]{YuditskiiAdv} we have
		\[
		p_{g-1}^{(-1)}(m)=-q_{g-1}^{(-1)}(m+1)f(m),
		\]
		where $f(m)$ is an explicite function that can be small only on a sparse set. 
		Combining this with 
		\[
		(\Lambda_{-2}(m)\lambda_g-1)p_{g-1}^{(-1)}(m)=-(\Lambda_{-2}(m)\lambda_g-1)q_{g-1}^{(-1)}(m+1)f(m)
		\]
		we also get an estimate for $\|\{(\Lambda_{-2}(m)\lambda_g-1)q_{g-1}^{(-1)}(m)\}_{m=1}^N\|_{\ell^2}$, which shows 
		\[
		\|\{(\Lambda_{-2}(m)\lambda_g-1) \}\|\leq  C(\tilde H_+(A_N)+|I_N|).
		\]
		The same arguments with respect to $\cO^k A$, $k=1,...,g-1$, in a combination with \eqref{eq:commJacandO}, yield the estimates for all other components of $F_\E$.
	\end{proof}

\begin{lemma}
	There exists $\{\e^\a_m\}\in\CS(\bbN,\bbR^g/\bbZ^g)$ and  $\{\e^a_m\}\in\CS(\bbN,\bbR), \{\e^b_n\}\in\CS(\bbN,\bbR)$ so that 
	\begin{align*}
	a_m^2=\cA\bigg(\sum_{j=1}^m\vare^\a_j-m\chi\bigg)+\vare^a_m,\\
	b_m=\cB\bigg(\sum_{j=1}^m\vare^\a_j-m\chi\bigg)+\vare^b_m,
	\end{align*}
	where $\cA,\cB$ are given in \eqref{eq:parameter}.
\end{lemma}
\begin{proof}
	Let $A(\vbp_0(m))$ be the periodic GMP matrix with coefficients $\vbp_0(m)$ and  $A(\a_m)\in\cT_\E(\bC_E)$, so that 
	\[
	\dist(\vbp_0(m),\cI\cS_\E)=\dist(\vbp_0(m),\mathring{\vbp}(\a_m)).
	\]
	Thus, using \eqref{eq:openGap} we obtain
	\[
	\sum_{m=1}^N\dist(\vbp_0(m),\mathring{\vbp}(\a_m))^2\leq C(\tilde H_+(A_N)+|I_N|)
	\]
	and by  \eqref{eq:abCoeff} we get
	\begin{align*}
	\sum_{m=1}^N(a_m^2-\cA(\a_m))^2&\leq C(\tilde H_+(A_N)+|I_N|),\\
	\sum_{m=1}^N(b_m- \cB(\a_m))^2&\leq C(\tilde H_+(A_N)+|I_N|),
	\end{align*}
	where again $\{a_m,b_m\}_{m\in\bbN_0}$ denote the coefficients of $J_+$.
	Thus, dividing by $N$ and sending $N\to\infty$, we obtain by \eqref{eq:3Jan20222}  and \eqref{eq:tildeHplusAN} that
	\begin{equation}\label{17aug3}
	\{a_m^2-\cA(\a_m)\}_{m\in\bbN_0},\{b_m-\cB(\a_m)\}_{m\in\bbN_0} \in\CS.
	\end{equation}
	The smoothness of the Jacobi flow transform, provided that $p_g^{(0)},p_g^{(1)}>\d$, allows for the definition of a sparse set $I_N$ so that 
	\begin{align*}
	\dist(\vbp_0(m+1),\mathring\vbp(\alpha_m-\chi))&=\dist(\cJ(\vbp_0(m),\vbp_1(m)),\cJ(\vbp(\alpha_m))\\
	&\leq
	C(\E,J,\d)\{
	\dist(\vbp_0(m),\mathring\vbp(\alpha_m))+
	\dist(\vbp_0(m),\vbp_1(m))+|I_N|\}.
	\end{align*}
	Thus,
	\begin{align*}
	\dist(\mathring\vbp(\alpha_{m+1}),\mathring\vbp(\alpha_m-\chi))
	&\le C(\E,J,\d)(
	\dist(\vbp_0(m),\mathring\vbp(\alpha_{m}))
	+\dist(\vbp_0(m),\vbp_1(m))+|I_N|)\\
	&+\dist(\vbp_0(m+1),\mathring\vbp(\alpha_{m+1})).
	\end{align*}
	Moreover, we have
	$$
	\|\alpha-\beta\|\le C_1(E)
	\dist(\mathring\vbp(\alpha),\mathring\vbp(\beta)).
	$$
	Thus, defining $\e_\a(m)=\a_{m+1}-(\a_m-\chi)$, we conclude from \eqref{eq:Jan14} that 
	\[
	\{\e_\a\}\in \CS(\bbN,\bbR^g/\bbZ^g).
	\]
\end{proof}

\begin{lemma}
	For fixed $L \in \bbN$ and $\delta > 0$, the set
	\[
	B_{L,\delta} = \bigg\{ m : \bigg\|\sum_{j=m+1}^{m+\ell} \e_j^\a  \bigg\| \le \delta \text{ for all }\ell =0,\dots, L-1 \bigg\}
	\]
	has a sparse complement, i.e., $\frac{ \lvert  B_{L,\delta} \cap \{ 1, \dots, N\} \rvert}N \to 1$ as $N \to \infty$.
\end{lemma}

\begin{proof}
	Since shifts and linear combinations of $\CS$ sequences are in $\CS$,  $\big\{\sum_{j=m+1}^{m+\ell} \e_j^\a \big\}_{m=0}^\infty  \in \CS$ for any $\ell$. Thus, for any $\ell$, the set $\big\{ m : \bigg\|\sum_{j=m+1}^{m+\ell} \e_j^\a  \bigg\| > \delta\big\}$ is sparse; the complement of $B_{L,\delta}$ is a union of finitely many sparse sets, so it is sparse.
\end{proof}

\begin{proof}[Proof of Theorem~\ref{conj1}]
It remains to prove that, for every $\epsilon > 0$,
\begin{equation}\label{17aug1}
\limsup_{N\to\infty} \frac{1}{N}\sum_{m=1}^N\dist(\cT_\E^+,(S_+^*)^mJ_+S_+^m) \le \epsilon.
\end{equation}
Fix $L$ so that $\sum_{\ell = L}^\infty e^{-\ell} \|J_+\|  \le \epsilon /16$.  Choose $\delta > 0$ so that
\begin{equation}\label{17aug2}
\lvert \cA(\b_1) - \cA(\b_2) \rvert \le \frac \epsilon{8 L} , \qquad \lvert \cB(\b_1) - \cB(\b_2) \rvert \le \frac \epsilon{8 L}
\end{equation}
whenever $\lvert \b_1 - \b_2 \rvert \le \delta$.

Since $\dist (\cT_\E^+,(S_+^*)^mJ_+S_+^m)$ is uniformly bounded in $m$ and the complement of  $B_{L,\delta}$ is sparse,
\[
\limsup_{N\to\infty} \frac{1}{N} \sum_{\substack{1\le m \le N \\ m \notin B_{L,\delta}}} \dist (\cT_\E^+,(S_+^*)^mJ_+S_+^m)  = 0.
\]
Set $\a_m=\sum_{j=1}^m\vare^\a_j$. For $m\in B_{L,\delta}$, estimating the distance to $\cT_\E^+$ by the distance to $J(\a_m - m\chi)$ gives
\[
\dist(\cT_\E^+,(S_+^*)^mJ_+S_+^m)  \le  \sum_{\ell=0}^{\infty} e^{-\ell}(|a_{m+\ell}-\cA(\a_{m}- (m+\ell)\chi)|+|b_{m+\ell}-\cB(\a_{m}-(m+\ell) \chi)|).
\]
Using \eqref{17aug2} for $\ell < L$ and using our choice of $L$ to bound the tail of the series, we obtain
\[
\dist(\cT_\E^+,(S_+^*)^mJ_+S_+^m)  \le  \frac \epsilon 2 +  \sum_{\ell=0}^{L-1} e^{-\ell}(|a_{m+\ell}-\cA(\a_{m+\ell}- (m+\ell)\chi)|+|b_{m+\ell}-\cB(\a_{m+\ell}-(m+\ell) \chi)|).
\]
Thus, to prove \eqref{17aug1}, it remains to prove
\begin{equation}\label{17aug5}
\limsup_{N\to\infty} \frac{1}{N} \sum_{\substack{1\le m \le N \\ m \in B_{L,\delta}}} \sum_{\ell=0}^{L-1} e^{-\ell} g_{m+\ell} \le \frac \epsilon 2,
\end{equation}
where $g_p = |a_p-\cA(\a_{p}- p \chi)|+|b_{p}-\cB(\a_{p}-p \chi)|$. Note $g \in \CS$ by \eqref{17aug3}.  Enlarging the range of summation, we obtain
\begin{align*}
\limsup_{N\to\infty} \frac{1}{N} \sum_{\substack{1\le m \le N \\ m \in B_{L,\delta}}} \sum_{\ell=0}^{L-1} e^{-\ell} g_{m+\ell}
& \le \limsup_{N\to\infty} \frac 1N  \sum_{p=1}^{N+L} \sum_{\ell=0}^{L-1} e^{-\ell} g_p.
\end{align*}
Now the sum in $\ell$ can be separated as an explicit constant, so this $\limsup$ is zero since $g \in \CS$. Then \eqref{17aug5} follows, and the proof of \eqref{17aug1} is complete.
\end{proof}
\providecommand{\MR}[1]{}
\providecommand{\bysame}{\leavevmode\hbox to3em{\hrulefill}\thinspace}
\providecommand{\MR}{\relax\ifhmode\unskip\space\fi MR }
\providecommand{\MRhref}[2]{%
	\href{http://www.ams.org/mathscinet-getitem?mr=#1}{#2}
}
\providecommand{\href}[2]{#2}

	\end{document}